\documentclass[11pt]{article}

\usepackage{hyperref}
\usepackage[normalem]{ulem}
\usepackage{amsthm}
\usepackage{amsmath}
\usepackage{amssymb}
\usepackage{bm}
\usepackage{mathrsfs}
\usepackage{amsfonts}
\usepackage{mathtools}
\usepackage{graphicx}
\usepackage{tikz}
\usetikzlibrary{arrows}
\usepackage{enumitem}
\usepackage{caption}
\usepackage{titlesec}
\usepackage{upgreek}
\usepackage{fullpage}
\usepackage{appendix}
\usepackage[capitalise]{cleveref}
\usepackage{natbib}
\usepackage{svg}

\titleformat{\section}[hang]{\normalfont\bfseries}{\thesection.}{.5em}{}[]
\titlespacing{\section}{\parindent}{3ex plus .1ex minus .2ex}{1em}

\titleformat{\subsection}[runin]{\normalfont\itshape}{\thesubsection.}{.5em}{}[.]
\titlespacing{\subsection}{\parindent}{3ex plus .1ex minus .2ex}{1em}

\newtheorem{theorem}{Theorem}

\newtheorem{lemma}[theorem]{Lemma}

\newtheorem{corollary}[theorem]{Corollary}

\theoremstyle{definition}
 
\theoremstyle{remark}

\theoremstyle{remark}

\numberwithin{theorem}{section}

\providecommand{\N}{}

\providecommand{\P}{}

\renewcommand{\N}{{\mathbb N}}

\renewcommand{\P}{\mathbb{P}}

\newcommand{\E}{\mathbb{E}}

\newcommand{\I}[1]{{\mathbf 1}_{[#1]}}

\newcommand{\bigO}{\mathcal{O}}

\let\savehline\hline
\def\hline{\noalign{\vskip0pt}\savehline\noalign{\vskip0pt}}

\newcommand{\Gnp}{\mathcal{G}(n,p)}
\newcommand{\ds}{\displaystyle}
\newcommand{\ER}{Erd\H{o}s-R\'e{}nyi}

\DeclareRobustCommand{\stirling}{\genfrac\{\}{0pt}{}}

\title{Direct Paths in the Temporal Hypercube}

\author{
Austin Eide\footnote{Department of Mathematics, Toronto Metropolitan University, Toronto, ON, Canada; e-mail: \texttt{austin.eide@torontomu.ca}}, 
Martijn Gösgens\footnote{Centrum Wiskunde \& Informatica, Amsterdam, the Netherlands; e-mail: \texttt{research@martijngosgens.nl}},
Pawe\l{} Pra\l{}at\footnote{Department of Mathematics, Toronto Metropolitan University, Toronto, ON, Canada; e-mail: \texttt{pralat@torontomu.ca}}
}

\begin{document}
\maketitle

\begin{abstract}
We consider the $n$-dimensional random \textit{temporal hypercube}, i.e., the $n$-dimensional hypercube graph with its edges endowed with i.i.d.\ continuous random weights. We say that a vertex $w$ is \emph{accessible} from another vertex $v$ if and only if there is a path from $v$ to $w$ with increasing edge weights. We study accessible ``direct" paths from a fixed vertex to its antipodal point and show that as $n \to \infty$, the number of such paths converges in distribution to a mixed Poisson law with mixture given by the product of two independent exponentials with rate $1$. Our proof makes use of the Chen-Stein method, coupling arguments, as well as combinatorial arguments which show that typical pairs of accessible paths have small overlap.
\end{abstract}

\section{Introduction}

In a temporal graph, each edge is endowed with a weight that we interpret as a time-stamp; an edge is inactive before the time-stamp is reached, after which point it is active. We say that a path in a temporal graph is \textit{accessible} if it is increasing with respect to the edge weights, and a vertex $w$ is accessible from $v$ if there is some accessible path from $v$ to $w$ in the graph. The graph is \textit{temporally connected} if all vertices are pairwise accessible from one another. Temporal graphs provide a natural and perhaps more realistic model for a variety of dynamic network processes. For instance, in a dynamically evolving social network, information can only pass from one node to another after they have formed a connection in the network. On the theoretical side, extending basic graph theoretic properties such as connectivity and diameter to the temporal setting is an interesting technical challenge. We refer to~\cite{michail2016introduction} for a broad overview of the topic.

In this paper, we are interested in the setting in which edge weights are assigned randomly. We remark that when edge weights are sampled i.i.d.\ from any continuous distribution, the model is equivalent to one in which the edges are ordered by a random permutation (an equivalence we exploit throughout the paper). Of particular interest in the literature has been the length of a longest accessible path in a randomly weighted temporal graph. Using the second moment method, \cite{lavrov2016increasing} showed that a randomly edge-ordered complete graph $K_{n}$ contains an accessible Hamiltonian path with probability at least $\frac{1}{e} + o(1)$, and additionally that it contains accessible paths of length $0.85n$ with high probability\footnote{An event holds \emph{with high probability} (\emph{w.h.p.}) if it holds with probability tending to one as $n \to \infty$.}. They conjecture that an accessible Hamiltonian path also exists w.h.p., a problem which remains open. We note that on the deterministic side, \cite{graham1973increasing} showed that for \textit{any} edge ordering of $K_{n}$ there exists an accessible path of length at least $\sqrt{n-1}$, and \cite{calderbank1984increasing} construct orderings for which the longest accessible path has length at most $(1+o(1))\frac{n}{2}$.

There are also significant recent results on temporal \emph{random} graphs~\citep{angel2020long,becker2022giant,broutin2023increasing,casteigts2024sharp}. For the temporal \ER~random graph $\Gnp$ with connection probability $p$, \cite{casteigts2024sharp} establish that there are several phase transitions: at $p= \frac {\log n}{n}$, any fixed pair of vertices is mutually accessible w.h.p.; at $p= \frac {2\log n}{n}$, for any fixed $v$, all vertices are accessible from $v$ w.h.p.; and at $p=\frac {3\log n}{n}$, the graph is temporally connected with high probability. \cite{angel2020long} shows that for $\frac{\log n}{n} \ll p \ll 1$ the longest accessible path in $\Gnp$ has length in $[(1-\varepsilon)enp, (1+\varepsilon)enp]$ w.h.p.\ for any fixed $\varepsilon >0$; \cite{broutin2023increasing} extend this result to the regime $p = \frac{c\log n}{n}$, showing (among other results) that the longest accessible path has length in $[(1-\varepsilon)\alpha(c)\log n, (1+\varepsilon)\alpha(c)\log n]$ w.h.p.\ for an explicit constant $\alpha(c)$. For the $d$-dimensional Random Geometric Graph, it was proven~\citep{Brandenberger_Donderwinkel_Kerriou_Lugosi_Mitchell_2025} that the threshold for temporal connectivity occurs when the degrees are $\Theta(n^{1/(d+1)})$, which is much higher than the $\Theta(\log n)$ threshold for the \ER~random graph.

It is easy to see that the hypercube cannot be w.h.p.\ temporarily connected. For example, we expect one vertex with all incident edges with weights more than $\tfrac12$ and one with all weights less than $\tfrac12$. With probability bounded away from zero we have at least one vertex of each type which implies that there is no accessible path from one of them to the other. Hence, in this work we study the following variant of the problems mentioned above. Consider the $n$-dimensional hypercube graph, with vertex set given by the power set of $[n] = \{1,2,\dots,n\}$ and edge set given by pairs $v,v' \subseteq [n]$ with $|v \Delta v'| = 1$. We say that a path from $v$ to $v'$ in the hypercube is \textit{direct} if it has length equal to the graph distance of $v$ and $v'$. Assign i.i.d.\ $\text{Unif}[0,1]$ weights to the edges, and define $X_{n}$ to be the random variable which counts the number of accessible direct paths from $\emptyset$ to $[n]$. We remark that $X_{n}$ is a particularly critical random variable: there are $n!$ direct paths from $\emptyset$ to $[n]$, and each of them is accessible with probability exactly $\frac{1}{n!}$, which means $\E[X_n]=1$. If the correlations between different paths are negligible, then it is reasonable to expect $X_n$ to converge to a Poisson random variable with parameter $1$. On the other hand, the presence of one directed path increases chances for other overlapping direct paths. If this correlation is strong, we may have $X_n=0$ w.h.p.\ and $\E[X_n\ |\ X_n>0]\to\infty$, so that $\E[X_n]=1$. These correlations, however, are highly non-trivial. Understanding this critical scenario was proposed as one of the open problems during the 17th Annual Workshop on Probability and Combinatorics, McGill University’s Bellairs Institute, Holetown, Barbados (April 7--14, 2023). Our main result, \cref{thm:main}, shows that both conjectures above are false: it establishes that $X_{n}$ converges in distribution to a mixed Poisson random variable $\text{Poi}(Z\cdot Z')$, where $Z$ and $Z'$ are independent exponential random variables of rate~$1$. In particular, $\mathbb{P}(X_n=0) \to \delta$, where $\delta=\int_{0}^{\infty}\frac{e^{-z}}{1+z}\,dz \approx 0.596347$ is the \emph{Gompertz} constant~\citep{gompertz1825mortality}.

While there is some previous work on the temporal hypercube (e.g.,~\cite{de2015increasing}), our result shares more significant similarities with the \textit{vertex-weighted} version of the problem, which has been studied extensively~\citep{berestycki2014accessibility,berestycki2016number,kauffman1987towards,kingman1978simple,martinsson2015accessibility}. Interest in the vertex-weighted hypercube comes from the field of evolutionary biology, where each vertex of the hypercube corresponds to a potential genetic code. One typically fixes the weight of vertex $[n]$ to $1$ and the weight of $\emptyset$ to $x\in[0,1)$, and considers the distribution of the number of accessible paths from one to the other. This process is sometimes also referred to as \emph{accessibility percolation}~\citep{martinsson2015accessibility}. As in the edge-weighted case, there are $n!$ direct paths from $\emptyset$ to $[n]$; in order for such a path to be increasing, all of the $n-1$ interior vertex weights must be at least $x$ and be in increasing order, which occurs with probability $\frac{(1-x)^{n-1}}{(n-1)!}$. Thus we expect $n(1-x)^{n-1}$ accessible paths. 

\cite{hegarty2014existence} establish a phase transition $x = \frac{\log n}{n}$, the point at which the expected number of increasing paths is $1 + o(1)$: for $\omega = \omega(n)$ tending to infinity arbitrarily slowly, if $x = \frac{\log n - \omega}{n}$ then there is no accessible direct path w.h.p., and if $x = \frac{\log n + \omega}{n}$ then an accessible direct path exists w.h.p. \cite{li2018phase} proves a similar result when backtracking is allowed. When $x = c/n$ for constant $c\geq 0$, \cite{berestycki2016number} show that the number of accessible direct paths from $\emptyset$ to $[n]$, rescaled by $n$, converges in distribution to $e^{-c}Z\cdot Z'$, where $Z$ and $Z'$ are independent exponential random variables of rate $1$. Our main result and proof techniques bear some similarity to this one, a connection we explore in more detail in~\cref{sec:proof_overview}.

Our proof strategy makes use of the \emph{Chen-Stein method}, which bounds the difference between the distribution of a sum of indicators and the Poisson distribution~\citep{arratia1990poisson}. The Chen-Stein method is typically used to prove Poisson limit laws, but has also been used to prove \emph{mixed} Poisson limits~\citep{reinert2007length} like our main result.

\subsection{Notation}

We consider a hypercube on $2^n$ vertices. Let $[n]=\{1,\dots,n\}$ denote the set of dimensions. Each of the $2^n$ vertices of the hypercube can be identified with a subset of $[n]$, corresponding to the set of nonzero coordinates of the vertex. This way, the vertex-set corresponds to the \emph{power set} $V=\mathcal{P}([n])$ of $[n]$. For $v\in V$, we refer to $|v|$ (the number of nonzero coordinates of $v$) as the \emph{level} of $v$, and we denote the set of vertices with level $k$ by $V_k$. We are interested in paths between the two endpoints $\emptyset$ and $[n]$. Two vertices $v,v'\subseteq [n]$ are connected if they differ by a single coordinate, i.e., if $|v\triangle v'|=1$, where $A \triangle B = (A \setminus B) \cup (B \setminus A)$ is the symmetric difference of two sets, $A$ and $B$. Because every vertex has degree $n$, the hypercube has $2^{n-1}n$ edges in total. If $v\subset v'\subseteq[n]$, then there exists a (not necessarily accessible) direct path from $v$ to $v'$. We denote the set of direct paths from $v$ to $v'$ by $\Pi(v,v')$. Each of these paths corresponds to a permutation of the directions $v'\setminus v$, so that $|\Pi(v,v')|=(|v'|-|v|)!$ if $v\subseteq v'$ and $|\Pi(v,v')|=0$ otherwise. For $U,U'\subseteq V$, we write $\Pi(U,U')=\bigcup_{u\in U,u'\in U'}\Pi(u,u')$ to denote the collection of direct paths from $U$ to $U'$.

We denote the set of hypercube edges by $E$. We say that an edge has level $k$ if it connects a vertex of level $k-1$ to a vertex of level $k$, and let $E_k$ denote the set of hypercube edges at level $k$. We let $W(e)$ denote the weight of edge $e\in E$. Each of the weights are independently drawn from a uniform distribution on $[0,1]$. For a path $\pi$, let $v_k(\pi)\in V_k$ denote the level-$k$ vertex of the path $\pi$. Similarly, let $W_k(\pi)$ denote the weight of the $k$-th edge along this path, and let $|\pi|$ denote the length of the path (i.e., the number of edges). We say that a path $\pi$ is \emph{accessible} if
\[
W_1(\pi)<\dots<W_{|\pi|}(\pi),
\]
and we denote the indicator of this event by $I(\pi)$. We are interested in the number of accessible paths from $\emptyset$ to $[n]$. That is, we study the distribution of
\[
X_n=\sum_{\pi\in\Pi(\emptyset,[n])}I(\pi).
\]
We use standard asymptotic notation throughout. For a real sequence $a_{n}$ and nonnegative real sequence $b_{n}$, we write $a_{n} = \bigO(b_{n})$ if $|a_{n}| \leq Cb_{n}$ for some $C>0$. Similarly, $a_{n} = \Omega(b_{n})$ if $|a_{n}| \geq cb_{n}$ for some $c > 0$ and $n$ sufficiently large. If both $a_{n} = \bigO(b_{n})$ and $a_{n} = \Omega(b_{n})$ hold, then we write $a_{n} = \Theta(b_{n})$.  If $\ds\lim_{n \to \infty}a_{n}/b_{n} = 0$, then we write $a_{n} = o(b_{n})$ or $a_n \ll b_n$. Finally, if $\ds\lim_{n \to \infty}b_{n}/a_{n} = 0$, then we write $a_{n} = \omega(b_{n})$ or $a_n \gg b_n$.

\subsection{Main results}

\begin{theorem}\label{thm:main}
    The number of accessible direct paths in the temporal hypercube converges in distribution to a mixed Poisson distribution, where the mixture corresponds to the product of two independent exponential random variables with rate $1$. That is,
    \[
    \P(X_n=x)\to\mathbb{E}\left[\frac{(Z\cdot Z')^x}{x!}e^{-Z\cdot Z'}\right],
    \]
    for $x\in\mathbb{N}$ and $Z,Z'\sim\text{Exp}(1)$ independently.
    More explicitly, the limiting distribution is given by
    \begin{equation}\label{eq:gompertz-expression}
    \P(X=x)=\sum_{k=0}^x{x\choose k}\frac{\delta-\sum_{r=0}^{k-1}(-1)^rr!}{k!},
    \end{equation}
    where $\P(X=0)=\delta=\int_{0}^{\infty}\frac{e^{-z}}{1+z}\,dz \approx 0.596347$ is the \emph{Gompertz} constant~\citep{gompertz1825mortality}.
\end{theorem}

\cref{thm:main} allows us to compute the limiting probability mass function for any $x$. The first few values are shown in~\cref{tab:pmf}.

\begin{table}[]
    \centering
    \begin{tabular}{c|c|c}
    $x$     &  $\P(X=x)$ & Rounded\\ \savehline
    $0$     & $\delta$ & $0.596347$ \\
    $1$     & $2\delta-1$ & $0.192695$ \\
    $2$     & $\frac72\delta-2$ & $0.087216$ \\
    $3$     & $\frac{17}3\delta-\frac{10}{3}$ & $0.045968$ \\
    $4$     & $\frac{209}{24}\delta-\frac{31}{6}$ & $0.026525$ \\
    $5$     & $\frac{773}{60}\delta-\frac{23}{3}$ & $0.016275$ \\
    $100$     &  & $1.78264 \cdot 10^{-9}$ \\
    $200$     &  & $3.85980 \cdot 10^{-13}$ \\
    $300$     &  & $5.97185 \cdot 10^{-16}$ \\
    \end{tabular}
    \caption{The first few values of the limiting probability distribution and a few larger ones; $\delta=\int_{0}^{\infty}\frac{e^{-z}}{1+z}\,dz$ is the \emph{Gompertz} constant~\citep{gompertz1825mortality}.}
    \label{tab:pmf}
\end{table}

Similarly to the result about the supercritical vertex-weighted hypercube from~\cite{berestycki2016number}, it turns out that the distribution of $X_n$ is largely determined by the weights close to the endpoints of the hypercube. The contribution of the weights close to the endpoint $\emptyset$ corresponds to the exponential random variable $Z$, while the contribution of the weights close to endpoint $[n]$ corresponds to $Z'$.

While~\cref{thm:main} is formulated in terms of the fixed vertices $\emptyset$ and $[n]$, it is easy to see that this result applies for \emph{any} two vertices $v,w$ at distance $d\le n$ as $d=d(n)\to\infty$, since all direct paths from $v$ to $w$ are contained in a $d$-dimensional hypercube. In particular, \cref{thm:main} also applies to $v,w$ chosen uniformly from an $n$-dimensional hypercube, since the distance between $v$ and $w$ is concentrated around $n/2$.

We note that the convergence in distribution of \cref{thm:main} does not necessarily imply convergence in moments, although we can obtain an explicit formula for the moments of the limiting distribution. Indeed, if $Y$ is a random variable and $X \sim \text{Poi}(Y)$, then by Proposition~1 of \cite{kuba2016moment}, the moments of $X$ are given by 
\[
    \E[X^{k}] = \sum_{j=0}^{k}\stirling{k}{j}\E[Y^{j}]
\]
where $\stirling{k}{j}$ is a \textit{Stirling number of the second kind}, that is, the number of ways to partition a set of $k$ objects into $j$ nonempty parts. In our case, the mixing distribution has moments given by $\E[(Z\cdot Z')^{k}] = \E[Z^{k}]^{2}= (k!)^{2}$, so that
\[
    \E[X^{k}] = \sum_{j=0}^{k}\stirling{k}{j}(j!)^{2}.
\]
The sequence above is the \textit{Stirling transform} of the sequence $(k!)^{2}.$ At present, we can only prove that $\E[X_{n}^{k}] \to \E[X^{k}]$ for $k \in \{1,2\}$. As mentioned above, $\E[X_{n}] = 1 = \E[X]$ for all $n$, and $\E[X_{n}^{2}] \to \sum_{j=0}^{2}\stirling{2}{j}(j!)^{2} = 5$, as summarized in the following lemma. 

\begin{lemma}\label{lem:2nd-moment}
    The second moment of $X_n$ converges to the second moment of $X\sim \text{Poi}(Z\cdot Z')$. That is,
    \[
    \mathbb{E}[X_n^2]\to \mathbb{E}[X^2]=5.
    \]
\end{lemma}

The limiting distribution is \emph{heavy-tailed}, in the sense that its tail decays slower than any exponential~\citep{nair2022fundamentals}. A nonnegative random variable is heavy-tailed whenever
\[
\lim_{x\to\infty}e^{tx}\P(X>x)\to\infty,
\]
for all $t>0$. For $X\sim \text{Poi}(Z\cdot Z')$, this can be seen from
\[
\P(X>x)\ge \P(Z,Z'>\sqrt{x})\cdot\P\left(X>Z\cdot Z'\ \middle|\ Z,Z'>\sqrt{x}\right)\sim \frac12e^{-2\sqrt{x}}=e^{-o(x)}.
\]

\subsection{Proof overview}\label{sec:proof_overview}

Here we give an overview of the proof of \cref{thm:main}, beginning with an observation which helps us understand the structure of the union of accessible direct paths from $\emptyset$ to $[n]$. A byproduct of our second moment computations in~\cref{sec:second_moment} and~\cref{sec:sm_appendix} is the following result, which we do not prove explicitly but which is easily derived from~\cref{lem:pair_sums}. Let $k = k(n) \leq \lfloor n/2 \rfloor$ grow to infinity arbitrarily slowly. W.h.p., any pair of accessible direct paths has the following structure: the paths travel together from $\emptyset$ until they diverge at some level  $0 \leq \ell \leq k$, and are edge-disjoint until they meet again at some level $n-k \leq \ell' \leq n$, after which point they travel together again to the terminal vertex $[n]$. In fact, a mild extension of our arguments shows that ``edge-disjoint" can be replaced with ``vertex-disjoint" in the preceding statement. The same structure was observed for the vertex-weighted case in~\cite{berestycki2016number} and~\cite{hegarty2014existence}, where similar second moment calculations are carried out.

Thus, the set of edges which are part of an accessible direct path from $\emptyset$ to $[n]$ induces a tree on the first $k$ levels and another tree on the last $k$ levels w.h.p., for any $k = k(n) \leq \lfloor n/2 \rfloor$ tending to infinity. Moreover, no accessible direct paths intersect between levels $k$ and $n-k$ w.h.p. This motivates splitting our analysis into two parts: one dealing with the ``tree segments" at the ends of the hypercube, and one dealing with the ``middle segment" in which leaves of the two trees are joined by compatible paths. We treat each part separately in Sections~\ref{sec:trees} and~\ref{sec:middle-part}.

Motivated by the observations above, for the tree segments our strategy is to greedily find a large tree in the first $k = k(n)$ levels of the hypercube, with root vertex $\emptyset$, such that every path from $\emptyset$ to a leaf is increasing, and such that the weights along the tree paths are small. We also find a symmetric counterpart to this tree in the last $k$ levels rooted at $[n]$, this time with all paths decreasing and consisting of large weights. Each non-leaf vertex in these trees has $r$ children. The children of a given vertex $v$ in the bottom tree are selected by choosing the $r$ directions with smallest weight which (i) have weight larger than the weight on the edge leading into $v$ in the tree, and (ii) have not already been chosen by an ancestor of $v$ in the tree. We use an analogous rule to construct the top tree. In~\cref{lem:poi-coupling}, we show that if $k \to \infty$ and $k^{4}r^{2}\log r \ll n$, then (a transformation of) the weights along the tree edges converge almost surely to sums of exponential random variables. A similar tree construction for the vertex-weighted hypercube appears in~\cite{berestycki2016number}, but instead of a fixed number of offspring they use a weight threshold to select children at each vertex. The fixed-offspring trees have additional structure, which we use in our proofs. We note that~\cite{berestycki2016number} proves a scaling limit while we characterize the exact limiting distribution, which requires more precise asymptotics for these tree segments.

For our purposes, choosing $r = \lceil \frac{3}{2}\log k\rceil$ will be enough to ensure that the trees are large enough to capture all accessible direct paths with high probability. With this choice for $r$, we let $\Pi^*_k\subset\Pi(\emptyset,[n])$ denote the set of paths where the first and last $k$ edges correspond to paths in the trees. Note that $\Pi_k^*$ is a random set, which depends only on the weights in the first and last $k$ levels. We then define the random variable 
\[
X_{k,n}^{*}=\sum_{\pi\in\Pi_k^*}I_\pi.
\]
Clearly, $X_{k,n}^{*}\le X_{n}$. We will prove in~\cref{sec:trees} that the paths outside $\Pi_k^*$ are negligible:
\begin{lemma}\label{lem:approximation}
    For $1\ll k\ll \frac{\log n}{\log\log n}\to\infty$, it holds w.h.p.\ that $X_n=X_{k,n}^{*}$.
\end{lemma}
Thus it will suffice to consider only the ``tree paths" counted by $X_{k,n}^{*}$. Let $E_{k}^{*}$ denote the set of edges in the first and last $k$ levels of the hypercube, that is,
\[
E^*_k=\bigcup_{\ell\in[k]} (E_{\ell}\cup E_{n-\ell+1}).
\]
Let $\lambda_{k,n} = \E[X_{k,n}^{*}|(W(e))_{e \in E_{k}^{*}}]$. (The definitions of the trees and the random variable $X_{k,n}^{*}$ will be made precise in \cref{sec:trees}.) We remark that $\lambda_{k,n}$ is a random variable which depends only on the edge weights in the first and last $k$ levels. The main result of \cref{sec:trees} is the following:

\begin{lemma}\label{lem:trees-main}
    There exists a coupling so that, $\lambda_{k,n} \xrightarrow{\P} ZZ'$ as $1\ll k\ll \frac{\log n}{\log\log n}\to\infty$, where $Z,Z'\sim\text{Exp}(1)$ independently.
\end{lemma}

The interpretation of \cref{lem:trees-main} is that each side of the hypercube contributes an exponentially distributed factor to the conditional expectation of $X_{k,n}^{*}$. 

Using the Chen-Stein method, we prove the following lemma about the middle segment in \cref{sec:middle-part}:

\begin{lemma}\label{lem:middle}
	Let $1\ll k\ll n$ and define $\lambda_{k,n}=\E\left[X_{k,n}^{*}\ \middle|\ (W(e))_{e\in E_k^*}\right]$. There exists a coupling between $X_{k,n}^{*}$ and a (mixed) Poisson random variable $Y_{k,n}$ with parameter $\lambda_{k,n}$, so that
    \[
    \P(X_{k,n}^{*}\neq Y_{k,n})\to0.
    \]
\end{lemma}

The above lemmas can be combined to prove our main result:

\begin{proof}[Proof of~\cref{thm:main}]
    We first use~\cref{lem:approximation} to write
    \[
    \P(X_n=x)=\P(X_n=x=X_{k,n}^{*})+\P(X_n=x,X_{k,n}^{*}\neq x)=\P(X_{k,n}^{*}=x)+o(1).
    \]
    Then, we use~\cref{lem:middle} to write
    \[
    \P(X_{k,n}^{*}=x)=\E\left[e^{-\lambda_{k,n}}\frac{\lambda_{k,n}^x}{x!}\right]+o(1).
    \]
    Since $e^{-\lambda}\lambda^x$ is bounded, we can apply~\cref{lem:trees-main} and the continuous mapping theorem to obtain
    \[
    \E\left[e^{-\lambda_{k,n}}\frac{\lambda_{k,n}^x}{x!}\right]\to\E\left[e^{-ZZ'}\frac{(ZZ')^x}{x!}\right],
    \]
    which completes the proof. The alternative expression~\eqref{eq:gompertz-expression} is derived in~\cref{lem:gompertz}. 
\end{proof}

\section{The second moment}\label{sec:second_moment}

Throughout this section, we let $\pi_{n} \in \Pi(\emptyset, [n])$ denote the \textit{canonical} direct path from $\emptyset$ to $[n]$ given by
$$
	\emptyset, [1], [2],\dots,[n].
$$
Consider any $\pi \in \Pi(\emptyset, [n]).$ We will encode $\pi$ relative to $\pi_{n}$. First, we specify the \textit{gap vector} $\bm{a}$ of $\pi$ with respect to $\pi_{n}$. If $\pi$ and $\pi_{n}$ do not share any edges, then $\bm{a}$ has only 1 coordinate, $a_0=n$. Assume then that these paths overlap at some point. The coordinates of $\bm{a}$ are defined as follows: let $a_{0}$ be the size of the gap, measured by number of edges of $\pi$ (or $\pi_n$), between vertex $\emptyset$ and the bottom vertex on the first edge $e_{1}$ which is shared by $\pi$ and $\pi_{n}$. For $i \geq 1$, let $a_{i}$ be the gap between the top vertex of the $i$th shared edge $e_{i}$ and the bottom vertex of the $(i+1)$st shared edge $e_{i+1}$, or $[n]$ if there are no more shared edges. If $\pi$ shares exactly $s$ edges with $\pi_{n}$, then $\bm{a}$ has exactly $s+1$ coordinates, and these coordinates sum to $n-s$. Moreover, $a_{i} \neq 1$ for all $i$, but we can have $a_{i} = 0$. In particular, $\pi_n$ itself has a gap vector of length $n+1$ and consists of only zeros. See Figure~\ref{fig:cube_paths} for an illustrated example.

\begin{figure}[h]
    \centering
    \includesvg[width=0.5\textwidth]{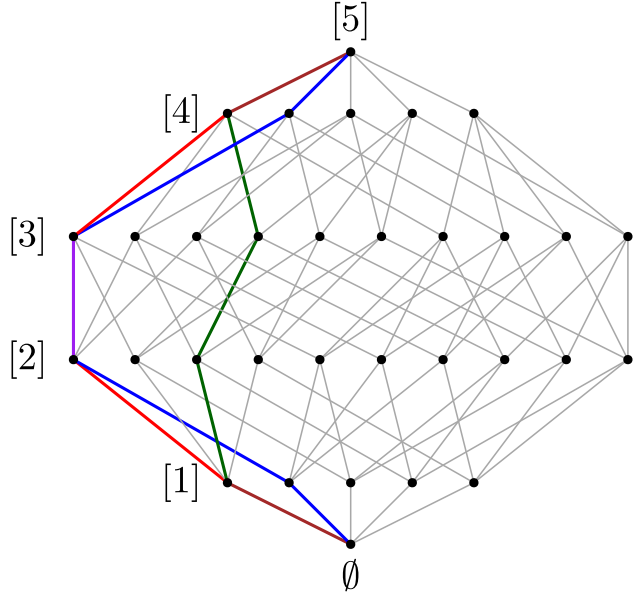}
    \caption{The hypercube with $n=5$ with its canonical path $\pi_{5}$ (red) along the left boundary. The blue path shares one edge (purple) with $\pi_{5}$ and has gap vector $\bm{a} = (2,2)$; the green path shares two edges (brown) with $\pi_{5}$ and has gap vector $\bm{a}=(0,3,0)$.}
    \label{fig:cube_paths}
\end{figure}

Second, we specify the subpaths that $\pi$ follows between its edge-intersections with $\pi_{n}$. For $i = 0,1,\dots$, the subpath of $\pi$ which traverses the $i$th gap is (isomorphic to) a path in an $a-i$-dimensional hypercube which is edge disjoint from the canonical path $\pi_{a_{i}}$. Thus we may represent the collection of gap subpaths with a tuple $\bm{G} = (\bar{\pi}_{0}, \bar{\pi}_{1},\dots,\bar{\pi}_{s})$, where $\bar{\pi}_{i}$ is a path of length $a_{i}$ from $\emptyset$ to $[a_{i}]$ in the $a_i$-dimensional hypercube, which is edge disjoint from $\pi_{a_{i}}$. The path $\pi$ can be completely recovered from its pair $(\bm{a},\bm{G})$. Moreover, it is clear that any such pair $(\bm{a},\bm{G})$ specifies a unique path $\pi$ which shares exactly $s$ edges with $\pi_{n}$. 

We can express the probability that a pair of paths is accessible using the gap vector $\bm{a}$.

\begin{lemma}\label{lem:int_probability}
Let $\pi$ be a path which has gap vector $\bm{a} = (a_{0},a_{1},\dots,a_{s})$ relative to $\pi_{n}$. Then the probability that both $\pi$ and $\pi_{n}$ are accessible is 
$$
	\E[I_{\pi}I_{\pi_{n}}] = \frac{\prod_{i=0}^{s}\binom{2a_{i}}{a_{i}}}{(2n-s)!}.
$$
\end{lemma}

\begin{proof}
The edge weights induce a random permutation of the edges of $\pi_{n} \cup \pi$, and either path is increasing with respect to the edge weights if and only if it is increasing with respect to the order of the corresponding permutation. There are $(2n-s)!$ possible permutations of the edges of $\pi_{n} \cup \pi$. Thinking of a permutation as a labeling of the edges with labels from the set $\{1,2,\dots,2n-s\}$, we observe that in any permutation for which both $\pi_{n}$ and $\pi$ are increasing, the labels of the $s$ shared edges are completely determined, i.e., there is exactly $1$ way these edges can be labeled. Similarly, the set of labels which can be assigned to edges in the $i$th gap (of both $\pi$ and $\pi_{n}$) is also determined for $i=0,1,\dots,s$. For each $i$, there are $2a_{i}$ labels to assign to the edges in gap $i$, and precisely $\binom{2a_{i}}{a_{i}}$ assignments result in both $\pi_{n}$ and $\pi$ being increasing across the gap. Indeed, once we distribute labels between $\pi_n$ and $\pi$, there is only one way to assign them to edges. Hence there are $\prod_{i=0}^{s}\binom{2a_{i}}{a_{i}}$ permutations for which both $\pi_{n}$ and $\pi$ are increasing.
\end{proof}

The nonzero coordinates of a gap vector $\bm{a} = (a_{0}, a_{1}, \dots,a_{s})$ correspond to the \textit{(nontrivial) gaps} between $\pi$ and $\pi_{n}$, i.e., segments where $\pi$ and $\pi_{n}$ diverge from one another. If $\pi$ has a gap vector of length $s+1$ with $g$ nonzero coordinates, we say that $\pi$ and $\pi_{n}$ share $s$ edges with $g$ gaps. More generally, we may speak of arbitrary pairs of paths $(\pi,\pi')$ from $\Pi(\emptyset,[n])$ which share $s$ edges with $g$ gaps. For a path $\pi$, we define
$$
	\mathcal{P}_{s,g}(\pi) = \{\pi'\,:\,\pi \in \Pi(\emptyset,[n]),\,\pi\text{ and }\pi'\text{ share }s\text{ edges with }g\text{ gaps}\}
$$
and we let 
\begin{equation}\label{eq:Psg}
	\mathcal{P}_{s,g} = \bigcup_{\pi \in \Pi(\emptyset,[n])}\bigcup_{\pi' \in \mathcal{P}_{s,g}(\pi)}\Big\{(\pi,\pi') \Big\}.
\end{equation}
So $\mathcal{P}_{s,g}$ is the set of all pairs of paths $(\pi,\pi')$ which share $s$ edges with $g$ gaps. Note that $\mathcal{P}_{s,g}$ is nonempty precisely when $s \in \{0,1,2,\dots,n-2\}$ and $1 \leq g \leq \min\left\{s+1, \left\lfloor\frac{n-s}{2}\right\rfloor \right\}$, or $s = n$ and $g = 0$. (It is impossible for a pair of paths to share exactly $n-1$ edges.) 

For $s \in \{0,1,\dots,n-2\}$ and $1 \leq g \leq\min\left\{s+1, \left\lfloor \frac{n-s}{2}\right\rfloor \right\}$, or $s = n$ and $g = 0$, define
\begin{equation}\label{eq:csg}
    c_{s,g} = \sum_{(\pi,\pi') \in \mathcal{P}_{s,g}}\E[I_{\pi}I_{\pi'}].
\end{equation}
Note that by symmetry we may write 
\begin{equation}\label{eq:sm_exp3}
    c_{s,g} = n!\sum_{\pi \in \mathcal{P}_{s,g}(\pi_{n})}\E[I_{\pi_{n}}I_{\pi}] = \sum_{\pi \in \mathcal{P}_{s,g}(\pi_{n})}\E[I_{\pi}\,|\,I_{\pi_{n}} = 1].
\end{equation}
Thus we have the following simple interpretation of the term $c_{s,g}$: conditioned on the path $\pi_{n}$ being accessible, we expect $c_{s,g}$ accessible paths which share $s$ edges and have $g$ gaps with $\pi_{n}$. \cref{lem:pair_sums} below computes the second moment of $X_n$ by carefully summing the terms $c_{s,g}$ over all valid $s$ and $g$. In light of the above, we may interpret \cref{lem:pair_sums} as stating that, conditioned on $\pi_{n}$ being accessible, for any $k = k(n)$ tending to infinity, in expectation there are $4 + o(1)$ accessible paths that share at most $k = k(n)$ edges with $\pi_{n}$ and have a single gap. Moreover, we show that, w.h.p., there are no accessible paths which have more than one gap, or which share more than $k$ edges with $\pi_{n}$.

\begin{lemma}\label{lem:pair_sums}
	The following three properties hold:
	\begin{enumerate}
		\item $\ds\sum_{s=1}^{n-2}\sum_{g=2}^{\min\left\{s+1, \left\lfloor\frac{n-s}{2}\right\rfloor \right\}}c_{s,g} = o(1)$;
		\item for any $k = k(n) \to \infty$, $\ds\sum_{s=k}^{n-2}c_{s,1} = o(1)$; and,
		\item for any $k = k(n) \to \infty$, $\ds\sum_{s=0}^{k}c_{s,1} = 4 + o(1)$.
	\end{enumerate}
\end{lemma}

The proof is heavily computational, so we postpone it to~\cref{sec:sm_appendix}. We remark that as a corollary, \cref{lem:pair_sums} implies \cref{lem:2nd-moment}: 
\begin{align*}
	\E[X_n^{2}] = c_{n,0} + \sum_{s=0}^{n-2}\sum_{g=1}^{\min\left\{s+1,\left\lfloor\frac{n-s}{2}\right\rfloor \right\}}c_{s,g} 
	= 1 + (4 + o(1)) + o(1) = 5 + o(1).
\end{align*}

\section{Tree segments}\label{sec:trees}

In this section, we will prove~\cref{lem:trees-main} in two main steps. First, we construct trees consisting of the edges that are most likely to be part of accessible paths near each endpoint. In~\cref{lem:exp-coupling}, we will introduce a coupling that proves that the weights of these tree edges converge almost surely to sums of exponential random variables. Second, we will construct two random variables $Z_k,Z_k'$ that are functions of the weights in the first and last $k$ levels, respectively. We will prove in \cref{lem:exp-limit} that $Z_k,Z_k'$ converge in probability to two independent exponentially distributed random variables. \cref{lem:trees} additionally proves $\lambda_{k,n}-Z_kZ_k'\stackrel{\P}\to0$, which will lead to the proof of \cref{lem:trees-main}.

\subsection{Tree coupling}

In this section, we construct a coupling between exponential random variables and the `best' weights in the first few levels. To do so, we first transform the uniform weights to exponential random variables. Notice that if $U\sim\text{Unif}([0,1])$, then
\(
-\log(1-U)\sim\text{Exp}(1).
\)
Let us denote these transformed weights by
\begin{equation}\label{eq:exp_transformation}
\widetilde{W}(e)=-\log(1-W(e)).
\end{equation}
A nice advantage of these transformed weights is that among $\ell$ edges, the smallest weight will follow the $\text{Exp}(\ell)$ distribution. Moreover, the difference between the smallest and second smallest weight is $\text{Exp}(\ell-1)$ and so forth. This can be seen from the fact that the remaining $\ell-1$ weights are each $\text{Exp}(1)$ distributed conditioned on being larger than the smallest weight. By the memorylessness property of the exponential distribution, the difference between each of these weights and the smallest weights is independently $\text{Exp}(1)$ distributed, so that the minimum of them is indeed $\text{Exp}(\ell-1)$ distributed.

To construct a tree from the most important edges, we use an approach similar to~\cite{berestycki2016number} and exclude directions that have been taken by an ancestor node. The main difference is that we consider a fixed number of offspring $r$, while they consider weight thresholds. For $i_1,\dots,i_\ell\in[r]$, we denote by $\pi(i_1,\dots,i_\ell;r)$ a path from $\emptyset$ to some $v\in V_\ell$ where, in the first step, we make a step in the direction with the $i_1$-th smallest weight. After that, the path makes a step in the direction of the $i_2$-th smallest weight that is larger than the current weight and excludes the $r$ directions that had the smallest weights in the first step. In the third step, we consider all weights larger than the current weight, excluding the $2r$ `best' directions of its ancestors and pick the $i_3$-rd smallest among those, and so forth. Note that such a path $\pi(i_1,\dots,i_\ell;r)$ may not exist for all $r,\ell$ and $i_1,\dots,i_\ell$. In particular, such a path does not exist if $(\ell-1)r>n$.

Let $\widetilde{W}_{n,r}(i_1,\dots, i_\ell)$ denote the weight of this path (i.e., $\widetilde{W}(e)$ of the last edge $e$ in this path) if the path $\pi(i_1,\dots,i_\ell;r)$ exists and $\widetilde{W}_{n,r}(i_1,\dots, i_\ell)=\infty$ otherwise. We will omit the subscript $n,r$ whenever the relevant values are clear from the context. We will abbreviate $\bm{i}=(i_1,\dots,i_\ell)$ and will denote the \emph{predecessor} of $\bm i\neq(1)$ as
\[
\bm i^-=\begin{cases}
    (i_1,\dots,i_{\ell-1},i_\ell-1)&\text{ if $i_\ell\ge2$,}\\
    (i_1,\dots,i_{\ell-1})&\text{ if $i_\ell=1$,}
\end{cases}
\]
and define the set of predecessors of $\bm i$ as
$$
P(\bm{i})=\{(i_1,\dots,i_{\ell'-1},j)\ :\ \ell'\in[\ell],j\in[i_{\ell'}]\}.
$$
We define
\(
\Delta \widetilde{W}_{n,r}(\bm i)=\widetilde{W}_{n,r}(\bm i)-\widetilde{W}_{n,r}(\bm i^-),
\)
so that
\begin{equation}\label{eq:W(i)_sum}
\widetilde{W}_{n,r}(\bm{i})=\sum_{\bm{j}\in P(\bm{i})}\Delta \widetilde{W}_{n,r}(\bm{j}).
\end{equation}
Given $k$ and $r$, if the path corresponding to $\bm{i} = (i_{1},i_{2},\dots,i_{k}) \in [r]^{k}$ exists, then after transforming back to uniform edge weights using~\eqref{eq:exp_transformation}, the weight on the last edge $e$ of $\bm{i}$ is distributed as
\[
    W_{n,r}(\bm{i}) = 1 - \exp\left(-\sum_{\bm{j} \in P(\bm{i})}\Delta\widetilde{W}_{n,r}(\bm{j}) \right).
\]
We may built a symmetric tree rooted at $[n]$ in the top of the cube in a similar manner as above, except that at each step we choose the largest available weight rather than the smallest. By our construction, if the path $\bm{i} \in [r]^{k}$ exists in the top cube, then for the last edge $e$ on $\bm{i}$ we have
\[
	W_{n,r}(e) = \exp\left(-\sum_{\bm{j}\in P(\bm{i})} \Delta \widetilde{W}'_{n,r}(\bm{j}) \right)
\]
where the $\Delta\widetilde{W}_{n,r}'(\bm{j})$'s are the increments of the top tree. We remark that $\Delta\widetilde{W}_{n,r}'(\bm{j})$ has the same distribution as $\Delta\widetilde{W}_{n,r}(\bm{j})$ for all $\bm{j}$. Moreover, if $k \leq \lfloor n/2 \rfloor$, then the bottom tree and the top tree are independent, as they depend on disjoint edge sets. Most of our attention in this section is devoted to understanding the behavior of the sums $\sum_{\bm{j} \in P(\bm{i})}\Delta\widetilde{W}_{n,r}(\bm{j})$ for $\bm{i} \in [r]^{k}$. To this end, below we describe a useful coupling between for the increments $\Delta\widetilde{W}_{n,r}(\bm{j})$.

We will couple $\Delta \widetilde{W}_{n,r}(\bm{j})$ to independent $\text{Exp}(1)$ random variables $(\Delta(\bm{i}))_{\bm{i}\in\mathbb{N}^\infty}$, where $\mathbb{N}^\infty=\bigcup_{\ell\in\mathbb{N}}\mathbb{N}^\ell$ is the set of integer sequences. For $k,r\in\mathbb{N}$ our coupling is as follows. For $i_1\in[r]$, we couple
\[
\Delta \widetilde{W}_{n,r}(i_1)=\frac{1}{n-(i_1-1)}\Delta(i_1),
\]
(Note that $\frac{1}{n-(i_{1}-1)}\Delta(i_{1})$ is distributed as $\text{Exp}(n-(i_{1}-1))$ for any $i_{1} \in [r]$.) To couple $\Delta \widetilde{W}_{n,r}(\bm i)$ for $\bm i\in[r]^\ell$ with $\ell\ge2$, we first sample a binomial random variable for its parent $\bm i'=(i_1,\dots,i_{\ell-1})$ as
\begin{equation}\label{eq:coupling-N}
N_{n,r}(\bm i')\sim\text{Bin}(n-(\ell-1)r,1-e^{-\widetilde{W}_{n,r}(\bm i')}),
\end{equation}
and then couple
\begin{equation}\label{eq:coupling}
\Delta \widetilde{W}_{n,r}(\bm i)=\begin{cases}
    \frac{\Delta(\bm i)}{n-(\ell-1)r-(i_\ell-1)-N_{n,r}(\bm i')}&\text{ if $N_{n,r}(\bm i')<n-(\ell-1)r-(i_\ell-1)$,}\\
    \infty&\text{ else.}
\end{cases}
\end{equation}
(Note that the parent $\bm i'$ is not (necessarily) the same as the predecessor $\bm i^-$.)

\begin{lemma}\label{lem:exp-coupling}
    The coupling defined by~\eqref{eq:coupling-N} and~\eqref{eq:coupling} is valid.
\end{lemma}
\begin{proof}
    We start with level $\ell=1$ and consider $i_1\in[r]$. If $i_1=1$, then $\Delta \widetilde{W}(1)=\widetilde{W}(1)$ and
    \[
    \P(\widetilde{W}(1)\ge w)=e^{-wn},
    \]
    which corresponds to the tail probability of an exponential random variable with rate $n$, so that $n\Delta \widetilde{W}(1)\sim\text{Exp}(1)$. Now, for $i_1\ge 2$, we know that there are $n-(i_1-1)$ edges from $\emptyset$ to $V_1$ with weight larger than $\widetilde{W}(i_1-1)$, so that
    \[
    \P(\widetilde{W}(i_1)\ge \widetilde{W}(i_1-1)+ w)=e^{-w(n-(i_1-1))},
    \]
    so that indeed $(n-(i_1-1))\Delta \widetilde{W}(i_1)\sim \text{Exp}(1)$.

    For $\ell\ge2$, we first note that among the $n$ directions of the hypercube, $(\ell-1)r$ need to be excluded as each of the $\ell-1$ ancestors have expanded in $r$ unique directions that we exclude for the remainder of the subtree.
    Among the $n-(\ell-1)r$ remaining directions, we need to find the $i_\ell$-th smallest weight that is larger than $\widetilde{W}(i_1,\dots,i_{\ell-1})$. Let $N(i_1,\dots,i_{\ell-1})$ denote the number of edges that are smaller than this. Then, given $\widetilde{W}(i_1,\dots,i_{\ell-1})=w$, we have
    \[
    N(i_1,\dots,i_{\ell-1})\sim\text{Bin}(n-(\ell-1)r,1-e^{-w}).
    \]
    Then, given $N(i_1,\dots,i_{\ell-1})<n-(\ell-1)r-(i_\ell-1)$, $\widetilde{W}(i_1,\dots,i_\ell)$ is the $i_\ell$-th smallest among the $n-(\ell-1)r$ weights that are each at least $\widetilde{W}(i_1,\dots,i_{\ell-1})$, so that
    \[
        (n-(\ell-1)r-(i_\ell-1)-N(i_1,\dots,i_{\ell-1}))\cdot\Delta \widetilde{W}(i_1,\dots,i_\ell)\sim\text{Exp}(1),
    \]
    as required. Whenever $N(i_1,\dots,i_{\ell-1})\ge n-(\ell-1)r-(i_\ell-1)$, the path $\pi(i_1,\dots,i_\ell;r)$ does not exist, so that $\widetilde{W}(i_1,\dots,i_\ell)=\infty$.
\end{proof}

From~\eqref{eq:coupling}, it is clear that we have $n\widetilde{W}_{n,r}(\bm{i}) \geq \sum_{\bm{j} \in P(\bm{i})}\Delta(\bm{j})$ for all $\bm{i} \in [r]^{k}$. In the next lemma, we show a nearly-matching upper bound for the $n\widetilde{W}_{n,r}(\bm{i})$'s which holds with high probability provided $1 \ll r \ll k$.

\begin{lemma}\label{lem:poi-coupling}
    Let $1\ll r\ll k$. There exists a coupling and a sequence $n_{k,r}$, so that with probability at least $1-\frac1{k^2}$, 
    \[
    n\widetilde W_{n,r}(\bm i)\le \frac1k+\sum_{\bm j\in P(\bm i)}\Delta(\bm j)
    \]
    holds for all $\bm i\in[r]^k$ and $n\ge n_{k,r}$. Equivalently,
    \begin{equation}\label{eq:coupling-guarantee}
    \P\left(\exists \bm i\in [r]^k,n\ge n_{k,r}:\ n\widetilde W_{n,r}(\bm i)>\frac1k+\sum_{\bm j\in P(\bm i)}\Delta(\bm j)\right)<\frac1{k^2}.
    \end{equation}
\end{lemma}

\begin{proof}
    We will use induction on $k$ to prove that for any $k,r\in\mathbb N$ and $\varepsilon>0$, there is a random variable $M_{k,r}(\varepsilon)$ with $\P(M_{k,r}(\varepsilon)<\infty)=1$ so that
    \begin{equation}\label{eq:coupling-IH}
    n\ge M_{k,r}(\varepsilon)\Rightarrow \forall \bm i\in[r]^{\underline k}:\ n\Delta \widetilde{W}_{n,r}(\bm{i})\le \Delta(\bm i)+\varepsilon,
    \end{equation}
    where $[r]^{\underline k}=\bigcup_{\ell=1}^k[r]^\ell$.
    For $k=1$, we have
    \[
    n\Delta\widetilde W_{n,r}(i_1)=\frac{\Delta(i_1)}{1-\frac{i_1-1}{n}}=\Delta(i_1)+\frac{i_1-1}{n-(i_1-1)}\Delta(i_1)<\Delta(i_1)+\frac{r}{n-r}\max_{i_1\in[r]}\Delta(i_1).
    \]
    We solve
    \[
    \frac{r}{n-r}\max_{i_1\in[r]}\Delta(i_1)\le \varepsilon\Rightarrow n\ge r+\frac r\varepsilon\max_{i_1\in[r]}\Delta(i_1),
    \]
    so that the induction hypothesis~\eqref{eq:coupling-IH} holds for $k=1$ and any $r\in\mathbb N$ with the random variable $M_{1,r}(\varepsilon)=r+\frac r\varepsilon\max_{i_1\in[r]}\Delta(i_1)$.
    
    For the induction step, we assume~\eqref{eq:coupling-IH} holds up to $k$ and prove that it also holds for $k+1$. Let us pick $\bm i\in[r]^{k+1}$ and denote its parent by $\bm i'\in [r]^k$. For $n\ge M_{k,r}(\varepsilon)$, the induction hypothesis tells us that
    \[
    n\Delta \widetilde{W}_{n,r}(\bm{j})\le \Delta(\bm j)+\varepsilon
    \]
    for all $\bm j \in P(\bm i')\subset[r]^{\underline k}$. Summing over $\bm j \in P(\bm i')$, we obtain 
    \[
    n\widetilde{W}_{n,r}(\bm{i}')\le |P(\bm i')|\varepsilon+\sum_{\bm j\in P(\bm i')}\Delta(\bm i),
    \]
    so that the binomial random variable $N_{n,r}(\bm i')$ defined in~\cref{eq:coupling-N} is stochastically dominated by the Poisson random variable
    \[
    Y_{\varepsilon}(\bm i')\sim\text{Poi}\left(|P(\bm i')|\varepsilon+\sum_{\bm j\in P(\bm i')}\Delta(\bm j)\right).
    \]
    Note that $Y_{\varepsilon}(\bm i')$ does not depend on $n,k,r$. This means that for any $n\ge M_{k,r}(\varepsilon)$, we can couple $N_{n,r}(\bm i')$ to $Y_{\varepsilon}(\bm i')$ in a way that ensures $N_{n,r}(\bm i')\le Y_{\varepsilon}(\bm i')$.
    This means that
    \begin{eqnarray*}
    n\Delta \widetilde W_{n,r}(\bm i)
    &=\frac{\Delta(\bm i)}{1-\frac{kr+i_\ell-1}n-\frac1nN_{n,r}(\bm i')}
    &\le \frac{\Delta(\bm i)}{1-\frac{(k+1)r}n-\frac1nY_{\varepsilon}(\bm i')}\\
    &=\Delta(\bm i)+\frac{(k+1)r+Y_{\varepsilon}(\bm i')}{n-(k+1)r-Y_{\varepsilon}(\bm i')}\Delta(\bm i)
    &\le\Delta(\bm i)+\varepsilon,
    \end{eqnarray*}
    for $n\ge ((k+1)r+Y_{\varepsilon}(\bm i'))(1+\Delta(\bm i)/\varepsilon)$. This leads to
    \[
    M_{k+1,r}(\varepsilon)=\max\left\{M_{k,r}(\varepsilon),\max_{\bm i\in[r]^{k+1}}\left\{((k+1)r+Y_{\varepsilon}(\bm i'))(1+\Delta(\bm i)/\varepsilon)\right\}\right\},
    \]
    which is the maximum of a finite number of random variables, so that $\P(M_{k+1,r}(\varepsilon)<\infty)=1$, which completes the induction proof. 
    
    The random variable $M_{k,r}(\varepsilon)$ can be rewritten as follows:
    \[
    M_{k,r}(\varepsilon)=\max_{\ell\in[k],\bm i\in[r]^\ell}\left\{(\ell r+Y_{\varepsilon}(\bm i'))\left(1+\frac{\Delta(\bm i)}\varepsilon\right)\right\}.
    \]
    To complete the proof, we consider the random variable $M_{k,r}(k^{-2}r^{-1})$ and pick $n_{k,r}$ such that
    \[
    \P(M_{k,r}(k^{-2}r^{-1})>n_{k,r})<\frac1{k^2}.
    \]
    If there is some $n\ge n_{k,r}$ such that
    \[
    \exists \bm i\in [r]^k:\ n\widetilde W(\bm i)>\frac1k+\sum_{\bm j\in P(\bm i)}\Delta(\bm j),
    \]
    then there must be some $\bm j\in P(\bm i)$ with $n\Delta\widetilde W_{n,r}(\bm j)>\Delta(\bm j)+k^{-2}r^{-1}$. But by~\eqref{eq:coupling-IH}, this
    implies $M_{k,r}(k^{-2}r^{-1})>n$.
    Hence,
    \[
    \P\left(\exists \bm i\in [r]^k,n\ge n_{k,r}:\ n\widetilde W(\bm i)>\frac1k+\sum_{\bm j\in P(\bm i)}\Delta(\bm j)\right)\le \P(M_{k,r}(k^{-2}r^{-1})>n_{k,r})<\frac1{k^2},
    \]
    as stated in the lemma.
\end{proof}
The following lemma derives an admissible sequence $n_{k,r}$, which we prove in~\cref{sec:trees_appendix}.
\begin{lemma}\label{lem:coupling-admissible}
    For $1\ll r\ll k$, there exists a coupling and a sequence $n_{k,r}\sim 2k^4r^2\log r$ that achieves~\eqref{eq:coupling-guarantee}.
\end{lemma}
In particular, \cref{lem:coupling-admissible} tells us that $r_k=\lceil\frac{3}{2}\log k\rceil$ and $k\ll \frac{\log n}{\log\log n}$ is sufficient to ensure $n\gg n_{k,r}$. 

\subsection{Tree functionals}
In this section, we introduce the functional $Z_k=Z_{n,k,r_k}$, which is a function of the tree weights $(\widetilde{W}_{n,r_k}(\bm i))_{\bm i\in[r_k]^k}$ and its inverted counterpart $Z_k'$, which is a function of the tree weights from the other endpoint of the hypercube.
\begin{corollary}\label{lem:exp-limit1}
    Define
    \[
    Z_{n,k,r}=\sum_{\bm{i}\in [r]^k}e^{-n\widetilde{W}(\bm{i})},\quad\text{and}\quad Z_{\infty,k,r}=\sum_{\bm{i}\in [r]^k}\exp\left(-\sum_{\bm{j}\in P(\bm{i})}\Delta(\bm{j})\right).
    \]
    Then, there is some sequence $n_k\ge k$ as $k\to\infty$ so that
    \[
    \frac{Z_{n_k,k,r_k}}{Z_{\infty,k,r_k}}\stackrel{\text{a.s.}}\to1,
    \]
    as $k\to\infty$, where $r_k=\lceil\frac{3}{2}\log k\rceil$. Moreover, $\E[Z_{n_k,k,r_k}]\to1$.
\end{corollary}
\begin{proof}
    We pick $n_k=n_{k,r_k}$, for $n_{k,r}$ as given by~\cref{lem:poi-coupling}.
    This means that with probability at least $1-\frac1{k^2}$,
    \[
    0\le n\widetilde W_{n_k,r_k}(\bm i)-\sum_{\bm j\in P(\bm i)}\Delta(\bm j)\le \frac1k,
    \]
    holds for all $\bm i\in [r_k]^k.$
    By monototonicity of $Z_{n,k,r}$, we get
    \[
    e^{-\frac1k}Z_{\infty,k,r_k}\le Z_{n_k,k,r_k}\le Z_{\infty,k,r_k}.
    \]
    The remainder of the convergence proof follows from the Borel-Cantelli lemma.
    
    For the expectation, we use the dominated convergence theorem with bound $Z_{n_k,k,r_k}\le Z_{\infty,k,r_k}$ and compute
    \[
    \E[Z_{\infty,k,r_k}]=\sum_{i \in [r_k]}2^{-i}\E[Z_{\infty,k-1,r_k}],
    \]
    by independence of the $\Delta(\bm i)$ and $\E[e^{-\Delta(\bm i)}]=\frac12$. Solving this recursion leads to
    \[
    \E[Z_{\infty,k,r_k}]=\left(1-2^{-r}\right)^k=1-\bigO(k2^{-r_k}).
    \]
    Our choice $r_k=\lceil \frac32\log k\rceil$ ensures $k2^{-r_k}\le k^{1-\frac32\log2}\to0$ since $1-\frac32\log2\approx -0.0397<0$. 
\end{proof}

\begin{lemma}\label{lem:exp-limit2}
    There exists a random variable $Z\sim\textup{Exp}(1)$, so that
    \(
    Z_{\infty,k,r_k}\stackrel{\P}\to Z,
    \)
    as $k\to\infty$ and ${r_k=\lceil\frac32\log k\rceil}.$
\end{lemma}

\begin{proof}
For $i \in [r]$, define
\[
Z_{k-1,r}^{(i)}=\sum_{(i_2,\dots,i_k)\in[r]^{k-1}}\exp\left(-\sum_{\ell=2}^k\sum_{j=1}^{i_{\ell}}\Delta(i,i_2,\dots,i_{\ell-1},j)\right).
\]
Then note that $Z_{k-1,r}^{(1)},\dots, Z_{k-1,r}^{(r)}$ are independent copies of $Z_{\infty,k-1,r}$, since they are functions of disjoint sets of $\Delta(\bm i)$'s.
Moreover,
\[
Z_{\infty,k,r}= \sum_{i=1}^r e^{-\sum_{j=1}^i\Delta(j)}Z_{k-1,r}^{(i)}.
\]
This tells us that
\[
\E[Z_{\infty,k,r_k}]=(1-2^{-r_k})\E[Z_{\infty,k-1,r_k}]=(1-2^{-r_k})^k\to1.
\]

We now show that the limiting random variable
\(
Z=\lim_{k\to\infty}\lim_{r\to\infty}Z_{\infty,k,r}
\)
exists. 
To show that the inner limit exists, note that $[r]^k\subset [r+1]^k$, so that $Z_{\infty,k,r+1}-Z_{\infty,k,r}>0$. Moreover, the Markov inequality allows us to bound
\begin{eqnarray*}
\P\left(Z_{\infty,k,r+1}-Z_{\infty,k,r}>r^{-2}\right) &\le& r^2\E[Z_{\infty,k,r+m}-Z_{\infty,k,r}] \\
&=& r^2\left((1-2^{-r-1})^k-(1-2^{-r})^k\right)=\bigO(r^2k2^{-r}).
\end{eqnarray*}
This yields
\[
\sum_{r=1}^\infty\P\left(Z_{\infty,k,r+1}-Z_{\infty,k,r}>r^{-2}\right) = k\sum_{r=1}^\infty \bigO(r^22^{-r}) <\infty.
\]
We then use the Borel-Cantelli lemma to conclude that $Z_{\infty,k,r}$ converges almost surely as $r\to\infty$ to some random variable $Z_{\infty,k,\infty}$. 

To prove that $Z_{\infty,k,\infty}$ converges as $k\to\infty$, we show that it is a martingale:
\begin{align*}
    \E\left[Z_{\infty,k+1,\infty}\ \middle|\ \left(\Delta(\bm{i})\right)_{\bm{i}\in \bigcup_{1 \le\ell\le k}\mathbb{N}^\ell}\right]
    &=\sum_{\bm{i}\in \mathbb{N}^k}\exp\left(-\sum_{\bm{j}\in P(\bm{i})}\Delta(\bm{j})\right)\cdot\E\left[\sum_{a=1}^\infty\exp\left(-\sum_{1 \le b\le a}\Delta(\bm{i},b)\right)\right]\\
    &=\sum_{\bm{i}\in \mathbb{N}^k}\exp\left(-\sum_{\bm{j}\in P(\bm{i})}\Delta(\bm{j})\right)\cdot\sum_{a = 1}^{\infty}2^{-a}\\
    &=\sum_{\bm{i}\in \mathbb{N}^k}\exp\left(-\sum_{\bm{j}\in P(\bm{i})}\Delta(\bm{j})\right)\cdot1=Z_{\infty,k,\infty}.
\end{align*}
In the second line, we used that if $\Delta \sim \text{Exp}(1),$ then $\E[\exp(-\Delta)] = \frac{1}{2}$, and thus by independence $\E\left[\exp\left(-\sum_{b=1}^{a}\Delta(\bm{i}, b) \right)\right] = 2^{-a}$. Since $Z_{\infty,k,\infty}$ is a nonnegative martingale, Doob's martingale theorem ensures that it converges almost surely to some random variable $Z$.

Next, we show that $Z\sim\text{Exp}(1)$.
To do so, we first rewrite $Z_{\infty,k,\infty}$. Note that $1\in P(\bm{i})$ for all $\bm{i}\in\mathbb{N}^k$. Therefore,
\begin{align*}
Z_{\infty,k,\infty}
&=e^{-\Delta(1)}\sum_{\bm{i}\in\N^k}\exp\left(-\sum_{\bm{j}\in P(\bm{i})\setminus\{1\}}\Delta(\bm{j})\right)\\
&=e^{-\Delta(1)}\sum_{\bm{i}\in\N^k\ :\ i_1=1}\exp\left(-\sum_{\bm{j}\in P(\bm{i})\setminus\{1\}}\Delta(\bm{j})\right)+e^{-\Delta(1)}\sum_{\bm{i}\in\N^k\ :\ i_1>1}\exp\left(-\sum_{\bm{j}\in P(\bm{i})\setminus\{1\}}\Delta(\bm{j})\right)\\
&\stackrel{D}=e^{-\Delta(1)}Z_{\infty,k-1,\infty}+e^{-\Delta(1)}Z_{\infty,k,\infty}.
\end{align*}
Since $\Delta(1)\sim\text{Exp}(1)$, we have $e^{-\Delta(1)}\sim\text{Unif}([0,1])$. Taking the limit $k\to\infty$ on both sides tells us that
\[
Z\stackrel{D}=U\cdot(Z_{1}'+Z_{2}'),
\]
where $U\sim\text{Unif}([0,1])$ and $Z_{1}',Z_{2}'$ are independent copies of $Z$. We now derive the moment-generating function (MGF) of $Z$.
\begin{align*}
    M(t)=\E[e^{-tZ}]=\E\left[\E\left[e^{tU\cdot(Z_1'+Z_2')}\ \middle|\ Z_1',Z_2'\right]\right]=\E\left[\int_0^1e^{tu\cdot(Z_1'+Z_2')}du\right]=\E\left[\frac{e^{t\cdot(Z_1'+Z_2')}-1}{t\cdot(Z_1'+Z_2')}\right].
\end{align*}
Taking the derivative with respect to $t$ yields
\begin{align*}
    M'(t)
    =\E\left[\frac{(Z_1'+Z_2')\cdot e^{t\cdot(Z_1'+Z_2')}}{t\cdot(Z_1'+Z_2')}\right]-\E\left[\frac{e^{t\cdot(Z_1'+Z_2')}-1}{t^2\cdot(Z_1'+Z_2')}\right]=\frac{M(t)^2-M(t)}{t}.
\end{align*}
This is a separable differential equation:
$$
\frac{M'(t)}{M(t)^2-M(t)}=\frac1t.
$$
Integrating both sides yields
$$
\log(1-M(t))-\log M(t)=\log(t)+c\\
\Rightarrow \frac{1-M(t)}{M(t)}=te^c\\
\Rightarrow M(t)=\frac{1}{1+te^c}.
$$
The constant $e^c$ is determined by $1=\mathbb{E}[Z]=M'(0)=-e^{c}$, so that $e^c = -1$. We conclude that
\(
\E[e^{tZ}]=(1-t)^{-1},
\)
so that indeed $Z\sim\text{Exp}(1)$.

Finally, we prove $Z_{\infty,k,r_k}\stackrel\P\to Z$ for our choice of $r_k$.
Since $Z_{\infty,k,r}$ is increasing in $r$, we have $Z_{\infty,k,\infty}-Z_{\infty,k,r_k}>0$.
The Markov inequality allows us to bound
\[
\P(Z_{\infty,k,\infty}-Z_{\infty,k,r_k}>\varepsilon)\le\frac1\varepsilon\left(1-(1-2^{-r_k})^k\right)=\bigO(k2^{-r_k})\to0.
\]
This proves $Z_{\infty,k,\infty}-Z_{\infty,k,r_k}\stackrel\P\to0$, so that $Z_{\infty,k,r_k}\stackrel\P\to Z$.
\end{proof}

\begin{lemma}\label{lem:exp-limit}
    Consider $Z_k=Z_{n,k,r_k}$ and its inverted counterpart $Z_k'$ for $r_k=\lceil\frac32\log k\rceil$.
    There exists a coupling so that
    \[
        Z_k\stackrel{\P}\to Z,\quad\text{and}\quad Z_k'\stackrel{\P}\to Z',
    \]
    as $1\ll k\ll\frac{\log n}{\log\log n}\to\infty$,
    where $Z,Z'\sim\textup{Exp}(1)$ are independent exponential random variables.
\end{lemma}
\begin{proof}
    We take $Z_k=Z_{n_k,k,r_k}$ as defined in \cref{lem:exp-limit1}, which is a function of the weights of the first $k$ levels. For $Z_k'$, we take the inverted variant of $Z_{n,k,r_k}$. That is, we invert the hypercube by mapping every vertex $v\in V$ to its opposite $[n]\setminus v$, and replacing every edge weight $W(e)$ by $1-W(e)$. Let $Z'_{n,k,r_k}$ denote the function $Z_{n,k,r_k}$ applied to this inverted hypercube. Hence, $Z'_{n,k,r_k}$ is a function of the last $k$ levels of the original hypercube.

    We take the limit $Z$ from~\cref{lem:exp-limit2} and use~\cref{lem:exp-limit1} to write
    \[
    \frac{Z_k}{Z}=\frac{Z_{n_k,k,r_k}}{Z}=\frac{Z_{n_k,k,r_k}}{Z_{\infty,k,r_k}}\frac{Z_{\infty,k,r_k}}{Z}\stackrel{\P}\to1,
    \]
    so that indeed $Z_k\stackrel{\P}\to Z$.
    The proof for $Z_k'$ is identical.
\end{proof}
To prove~\cref{lem:approximation} and~\cref{lem:trees-main}, we will make use of the following lemma, which is proven in~\cref{sec:trees_appendix}. Recall that $\lambda_{k,n} = \E[X_{k,n}^{*}|(W(e))_{e \in E_{k}^{*}}]$.
\begin{lemma}\label{lem:trees}
Let $Z_k=Z_{n,k,r_k}$ as defined in \cref{lem:exp-limit1} and its inverted counterpart $Z_k'$. Then
    \[
    \E[|\lambda_{k,n} - Z_{k}Z_{k}'|] \to 0,
    \]
for $1\ll k\ll \frac{\log n}{\log\log n} \to\infty$.
\end{lemma}

\begin{proof}[Proof of~\cref{lem:approximation}]
    From $\E[Z_kZ_k']\to1$ (\cref{lem:exp-limit1}) and $\E[|\lambda_{k,n} - Z_{k}Z_{k}'|] \to 0$ (\cref{lem:trees}), it follows that $\E[\lambda_{k,n}]\to1$. This means that the nonnegative random variable $X_n-X_{k,n}^{*}$ has vanishing expectation, so that it follows from the Markov inequality that $X_n=X_{k,n}^{*}$ with high probability.
\end{proof}

\begin{proof}[Proof of~\cref{lem:trees-main}]
    The result follows directly by combining~\cref{lem:exp-limit} and~\cref{lem:trees}.
\end{proof}

\section{Middle part}\label{sec:middle-part}
To prove the coupling result from \cref{lem:middle}, we use the Chen-Stein method~\citep{arratia1990poisson} to prove that there exists a coupling between $Y_{k,n}$ and $X_{k,n}^{*}$ where $\P(Y_{k,n}\neq X_{k,n}^{*})$ is small. The Chen-Stein method is typically formulated in terms of the \emph{Total Variation} distance. However, for our purpose, it is not necessary to formally introduce the concept of total variation. 
We use the following `local' version of the Chen-Stein method, as presented in~\cite{ross2011fundamentals}:
\begin{theorem}[Theorem 4.7 of~\cite{ross2011fundamentals}]\label{thm:chen-stein}
    Let $(I_\theta)_{\theta\in\Theta}$ be a set of indicators, and let $\Gamma(\theta)\subset\Theta$ be the set of indicators that are dependent on $I_\theta$. Define $X=\sum_{\theta\in\Theta}I_\theta$, $\lambda=\mathbb{E}[X]$ and $Z\sim\text{Poi}(\lambda)$. There exists a coupling between $X$ and $Z$ so that
    \[
    \P(X\neq Z)\le\min\left\{1,\frac{1}{\lambda}\right\}\sum_{\theta\in\Theta}\sum_{\theta'\in \Gamma(\theta)}\left(\mathbb{E}[I_\theta]\mathbb{E}[I_{\theta'}]+\I{\theta\neq\theta'}\mathbb{E}[I_\theta I_{\theta'}]\right).
    \]
\end{theorem}

Let us define $\Pi^*_{k}$ as the set of paths $\pi$ from $\emptyset$ to $[n]$ so that the first $k$ steps of $\pi$ are contained in the $r_k$-ary tree rooted at $\emptyset$ and the last $k$ steps are contained in the tree rooted at $[n]$.
We consider
\[
X_{k,n}^{*}=\sum_{\pi\in \Pi^*_{k}}I_\pi\le X_n.
\]
To simplify the notation, we abbreviate
    \[
        \P_k(\cdot)=\P(\cdot|(W(e))_{e\in E_k^*}),\quad\text{and}\quad 
        \E_k[\cdot]=\E[\cdot|(W(e))_{e\in E_k^*}].
    \]

\begin{proof}[Proof of~\cref{lem:middle}]
    We consider the equivalent statement
    \[
    \E\left[\P_k(X_{k,n}^{*}\neq Y_{k,n})\right]\to0,
    \]
    which we prove by bounding $\P_k(X_{k,n}^{*}\neq Y_{k,n})$ using the Chen-Stein method.
To apply \cref{thm:chen-stein}, we need to construct a set of dependent events for each $\pi$. Two indicators $I_\pi,I_{\pi'}$ are dependent in $\P_k$ if they share an edge in the middle levels. For each $\pi\in\Pi^*_{k}$, we define $\Gamma(\pi)$ as the set of such overlapping paths.
    Applying \cref{thm:chen-stein} yields
    \begin{align}
    \P_k(X_{k,n}^{*}\neq Y_{k,n})
    &\le \min\left\{1,\frac1{\lambda_{k,n}}\right\}
    \sum_{\pi\in\Pi^*_k}\sum_{\pi'\in\Gamma(\pi)}\left(\E_k[I_\pi]\E_k[I_{\pi'}]+\I{\pi\neq\pi'}\E_k[I_\pi I_{\pi'}]\right)\nonumber\\
    &\le \sum_{\pi\in\Pi^*_k}\sum_{\pi'\in\Gamma(\pi)}\left(\frac1{\lambda_{k,n}}\E_k[I_\pi]\E_k[I_{\pi'}]+\I{\pi\neq\pi'}\E_k[I_\pi I_{\pi'}]\right).\label{eq:chen-stein}
    \end{align}
    We start by bounding the first part of the sum. We write
    \[
    \sum_{\pi\in\Pi^*_k}\sum_{\pi'\in\Gamma(\pi)}\frac1{\lambda_{k,n}}\E_k[I_\pi]\E_k[I_{\pi'}]\le\max_{\pi\in\Pi^*}\sum_{\pi'\in\Gamma(\pi)}\E_k[I_{\pi'}].
    \]
    Let $v_k(\pi)\in V_k$ denote the $k$-th vertex of $\pi$. Define
    \[
    V_k^*=\{v_k(\pi)\ :\ \pi\in\Pi_{k}^*\},
    \]
    and define $V_{n-k}^*$ similarly. For a path $\pi$, $\E_k[I_{\pi}]$ depends on the gap between the tree weights. This gap only depends on the endpoints $u=v_k(\pi)$ and $u'=v_{n-k}(\pi)$. Let us denote this weight gap by $w(u,u')$. Then
    \[
    \E_k[I_{\pi}]=\frac{w(u,u')^{n-2k}}{(n-2k)!}.
    \]
    For $u\in V_k,u'\in V_{n-k}$, we define
    \[
    \Gamma(\pi;u,u')=\{\pi'\in \Gamma(\pi)\ :\ v_k(\pi')=u,v_{n-k}(\pi')=u'\}.
    \]
    This allows us to write
    \[
    \E_k\left[\sum_{\pi'\in\Gamma(\pi)}I_{\pi'}\right]=\sum_{u\in V_k^*}\sum_{u'\in V_{n-k}^*}w(u,u')^{n-2k}\frac{|\Gamma(\pi;u,u')|}{(n-2k)!}.
    \]
    We now bound $|\Gamma(\pi;u,u')|/(n-2k)!$ uniformly over $u\in V_k^*,u'\in V_{n-k}^*$, and $\pi \in \Pi_{k}^{*}$. We do this by a union bound. Let $\Gamma_\ell(\pi;u,u')$ be the subset of $\Gamma(\pi;u,u')$ containing the paths that share an edge at level $k+\ell$ with $\pi$. By the union bound,
    \[
    |\Gamma(\pi;u,u')|\le\sum_{\ell=1,\dots,n-2k}|\Gamma_\ell(\pi;u,u')|.
    \]
    Then
    \begin{align*}
    	|\Gamma_\ell(\pi;u,u')|
    	&=(\ell-1)!\cdot (n-2k-\ell)!\I{w(u,u')>0,u\subset v_{k+\ell}(\pi) \subset u'}\\
    	&\le (\ell-1)!\cdot (n-2k-\ell)!\\
    	&=\frac{(n-2k-1)!}{{n-2k-1\choose \ell-1}}.
    \end{align*}
    This leads to
    \[
    \frac{|\Gamma(\pi;u,u')|}{(n-2k)!}\le\frac1{n-2k}\sum_{\ell=1}^{n-2k}{n-2k-1\choose \ell-1}^{-1}\sim \frac2n.
    \]
    Thus, 
    \begin{align*}
    \E_k\left[\sum_{\pi'\in\Gamma(\pi)}I_{\pi'}\right]
    &=\sum_{u\in V_k^*}\sum_{u'\in V_{n-k}^*}w(u,u')^{n-2k}\frac{|\Gamma(\pi;u,u')|}{(n-2k)!}\\
    &\le \frac{\E_k[X_{k,n}^{*}]}{n-2k}\sum_{\ell=0}^{n-2k-1}{n-2k-1\choose \ell}^{-1}\\
    &\sim\frac{2\lambda_{k,n}}{n}.
    \end{align*}
    Since this bound holds for all $\pi\in\Pi^*_k$ uniformly, we have obtained
    \[
        \max_{\pi\in\Pi^*}\sum_{\pi'\in \Gamma(\pi)}\E_k[I_{\pi'}]=\bigO\left(\frac{\lambda_{k,n}}{n}\right).
    \]
    We now bound the second part of the sum in~\eqref{eq:chen-stein}.
    We first take the expectation:
    \[
    \E\left[\sum_{\pi\in\Pi^*_k}\sum_{\pi'\in\Gamma(\pi)}\I{\pi\neq\pi'}\E_k[I_\pi I_{\pi'}]\right]=\sum_{\pi\in\Pi(\emptyset,[n])}\sum_{\pi'\in\Pi(\emptyset,[n])}\I{\pi\neq\pi'}\E[I_\pi I_{\pi'}\I{\pi\in\Pi^*_k,\pi'\in\Gamma(\pi)}].
    \] 
    The condition $\pi'\in\Gamma(\pi)$ implies that $\pi$ and $\pi'$ share an edge in $\bigcup_{\ell=k+1}^{n-k}E_\ell$.
    This implies that $\pi$ and $\pi'$ have $g \ge 2$ gaps or that they share more than $k$ edges. (See \cref{sec:second_moment}.) Therefore, we can bound by simply summing over all pairs
    \[
    (\pi,\pi')\in\mathcal P_k^*=\left(\bigcup_{\substack{s\ge 1 \\ g\ge2}}\mathcal P_{s,g}\right)\cup\left(\bigcup_{\substack{n-2 \geq s>k \\ g\ge1}}\mathcal P_{s,g}\right),
    \]
    where $\mathcal{P}_{s,g}$ is the set of ordered pairs of paths sharing exactly $s$ edges with $g$ gaps as defined in \eqref{eq:Psg}. We also recall the definition $c_{s,g} = \sum_{(\pi,\pi') \in \mathcal{P}_{s,g}}\E[I_{\pi}I_{\pi'}]$ in \eqref{eq:csg}. We have the bound
    \begin{align*}
        &\sum_{\pi\in \Pi(\emptyset,[n])}\sum_{\pi'\in\Pi(\emptyset,[n])}\I{\pi\neq\pi'}\E[I_\pi I_{\pi'}\I{\pi\in\Pi^*_k,\pi'\in\Gamma(\pi)}]\\
        \le&\sum_{(\pi,\pi')\in\mathcal P_k^*}\E[I_\pi I_{\pi'}]
        =\sum_{s=k+1}^\infty c_{s,1}+\sum_{g=2}^\infty\sum_{s=1}^\infty c_{s,g}
        =o(1),
    \end{align*}
    by~\cref{lem:pair_sums}.
    In conclusion, we can bound
    \begin{align*}
        \P(X_{k,n}^{*}\neq Y_{k,n})
        &\le \bigO\left(\frac{\E[\lambda_{k,n}]}{n}\right)+o(1)\to0, 
    \end{align*}
    which finishes the proof of the theorem. 
\end{proof}

\section{Acknowledgment}

This problem was proposed by G{\'a}bor Lugosi during the 17th Annual Workshop on Probability and Combinatorics, McGill University’s Bellairs Institute, Holetown, Barbados (April 7-14, 2023).
Martijn Gösgens was supported by the Netherlands Organisation for Scientific Research (NWO) through the Gravitation {\sc NETWORKS} grant no.\ 024.002.003.

\bibliographystyle{plain}
\bibliography{refs}

\appendix

\section{Proof of \cref{lem:pair_sums}}\label{sec:sm_appendix}
Here we prove \cref{lem:pair_sums}. We break up the proof into a series of smaller lemmas for readability. We begin by deriving an explicit expression for the term $c_{s,g}$. For $n \in \N_{0}$, let $w_{n}$ denote the number of paths in $\Pi(\emptyset, [n])$ which are edge disjoint from $\pi_{n}$. We take $w_{0} = 1$ by convention, $w_{1} = 0$, and in general we have $w_{n} \leq n!$. For $N, j, b \in \N_{0}$, define
$$
	\mathcal{C}(N,j;b) = \left\{(x_{1},x_{2},\dots,x_{j}) \in \N_{0}^{j}\,:\,\sum_{i=1}^{j}x_{i} = N,\,x_{i} \geq b\,\forall\,i\right\}.
$$
In other words, $\mathcal{C}(N,j;b)$ is the set of natural number compositions of $N$ with $j$ parts of at least $b$. We write $\mathcal{C}(N,j)$ for $\mathcal{C}(N,j;0)$. For future reference, we observe
$$
	|\mathcal{C}(N,j;b)| = \binom{N - (b-1)j -1}{j-1}
$$
(where the binomial coefficient on the right is $0$ whenever $N - (b-1)j -1 < j-1$). 

\begin{lemma} \label{lem:c_sg}
For $s \in \{0,1,\dots,n-2\}$ and $1 \leq g \leq \min\left\{s+1, \left\lfloor \frac{n-s}{2}\right\rfloor \right\}$,
	$$
		c_{s,g}=\binom{s+1}{g}\frac{n!}{(2n-s)!}\sum_{\bm{x} \in \mathcal{C}(n-s,g;2)}\prod_{i=1}^{g}w_{x_{i}}\binom{2x_{i}}{x_{i}}.
	$$ 
\end{lemma}

\begin{proof}
	We encode paths $\pi \in \mathcal{P}_{s,g}(\pi_{n})$ by the pairs $(\bm{a}, \bm{G})$ as outlined at the beginning of~\cref{sec:second_moment}. By definition, the vector $\bm{a}$ has length $s+1$ and exactly $g$ nonzero coordinates, all at least $2$, summing to $n-s$. To specify one uniquely we may thus choose a composition $\bm{x} = (x_{1}, x_{2},\dots,x_{g}) \in \mathcal{C}(n-s,g;2)$, then choose $g$ coordinates out of $s+1$ total in $a$ for $\bm{x}$ to occupy. If some $\pi \in \mathcal{P}_{s,g}(\pi_{n})$ has gap vector $\bm{a} = (a_{0},a_{1},\dots,a_{s})$ with nonzero coordinates $(x_{1},x_{2},\dots,x_{g}) \in \mathcal{C}(n-s,g;2)$, then by \cref{lem:int_probability} we have
	$$
		\E[I_{\pi_{n}}I_{\pi}] = \frac{1}{(2n-s)!}\prod_{i=0}^{s}\binom{2a_{i}}{a_{i}} = \frac{1}{(2n-s)!}\prod_{i=1}^{g}\binom{2x_{i}}{x_{i}}
	$$
	since $\binom{2a_{i}}{a_{i}} = 1$ if $a_{i} = 0$. Moreover, given $\bm{a}$, there are exactly 
	$$
		\prod_{i=0}^{s}w_{a_{i}} = \prod_{i=1}^{g}w_{x_{i}}
	$$
	choices for $\bm{G}$. It follows that
	$$
		\sum_{\pi \in \mathcal{P}_{s,g}(\pi_{n})}\E[I_{\pi_{n}}I_{\pi}] = \binom{s+1}{g}\frac{1}{(2n-s)!}\prod_{i=1}^{g}w_{x_{i}}\binom{2x_{i}}{x_{i}},
	$$
	and the proof of the lemma is finished.
\end{proof}

We also establish a convenient upper bound on $c_{s,g}$ which we will use repeatedly.

\begin{lemma}\label{lem:djk}
	For any $0 \leq s \leq n-2$ and $1 \leq g \leq \min\left\{s+1, \left\lfloor \frac{n-s}{2} \right\rfloor \right\}$, we have
	$$
		c_{s,g} \leq \frac{\binom{2(n-s)}{n-s}}{\binom{2n-s}{n}}\binom{s+1}{g}\frac{2^{g-1}}{(g-1)!(n-s-g)^{g-1}}
	$$
\end{lemma}

\begin{proof}
    We begin by writing 
	\begin{eqnarray}\label{eq:sum_rewrite}
            \sum_{\bm{x} \in \mathcal{C}(n-s,g;2)}\prod_{i=1}^{g}w_{x_{i}}\binom{2x_{i}}{x_{i}} &\leq& \sum_{\bm{x} \in \mathcal{C}(n-s,g;2)}\prod_{i=1}^{g}(x_{i})!\binom{2x_{i}}{x_{i}}\nonumber\\
            &=&(n-s)!\sum_{\bm{x} \in \mathcal{C}(n-s,g;2)}\binom{n-s}{x_{1},x_{2},\dots,x_{g}}^{-1} \prod_{i=1}^{g}\binom{2x_{i}}{x_{i}}
	\end{eqnarray}
	where we use $w_{x} \leq x!$ for any $x$ in the first line. It is straightforward to show that 
	\begin{equation}\label{eq:binomial_max}
            \max_{\bm{x} \in \mathcal{C}(n-s,g;2)}\prod_{i=1}^{g}\binom{2x_{i}}{x_{i}} \leq \max_{\bm{x} \in \mathcal{C}(n-s,g)}\prod_{i=1}^{g}\binom{2x_{i}}{x_{i}} = \binom{2(n-s)}{n-s},
	\end{equation}
	i.e., the maximum of the product in the middle above occurs at the point $\bm{x} = (n-s,0,0,\dots,0) \in \mathcal{C}(n-s,g)$. Similarly, 
	$$
            \min_{\bm{x} \in \mathcal{C}(n-s,g;2)}\binom{n-s}{x_{1},x_{2},\dots,x_{g}} = \binom{n-s}{n-s-2(g-1),2,\dots,2} = \frac{(n-s)!}{2^{g-1}(n-s-2(g-1))!}.
	$$
	Since $|\mathcal{C}(n-s,g;2)| = \binom{n-s-g-1}{g-1}$, we see
	\begin{eqnarray}\label{eq:multinomial_bound}
            \sum_{\bm{x} \in \mathcal{C}(n-s,g;2)}\binom{n-s}{x_{1},x_{2},\dots,x_{g}}^{-1} &\leq& \binom{n-s-g-1}{g-1} \cdot \frac{2^{g-1}(n-s-2(g-1))!}{(n-s)!}\nonumber\\
            &\leq& \frac{2^{g-1}}{(g-1)!} \cdot \frac{(n-s-g-1)!}{(n-s)!}\cdot\frac{(n-s-2(g-1))!}{(n-s-2g)!} \nonumber\\
            &\leq& \frac{2^{g-1}}{(g-1)!}\cdot\frac{(n-s-2g+2)(n-s-2g+1)}{(n-s-g)^{g+1}}\nonumber\\
            &\leq& \frac{2^{g-1}}{(g-1)!}\frac{1}{(n-s-g)^{g-1}} \ .
	\end{eqnarray}
	For the final inequality we need that $g \geq 2$, which implies $n-s-2g+1, n-s-2g+2 \leq n-s-g$. Even so, the inequality \eqref{eq:multinomial_bound} clearly still holds (with equality) for $g =1$, since in this case there is only the single summand $\binom{n-s}{n-s}^{-1} = 1$. Substituting the bounds \eqref{eq:binomial_max} and \eqref{eq:multinomial_bound} into \eqref{eq:sum_rewrite} yields 
	$$
            \sum_{\bm{x} \in \mathcal{C}(n-s,g;2)}\prod_{i=1}^{g}w_{x_{i}}\binom{2x_{i}}{x_{i}} \leq (n-s)!\binom{2(n-s)}{n-s}\frac{2^{g-1}}{(g-1)!(n-s-g)^{g-1}}
	$$
	from which point the result easily follows using the expression for $c_{s,g}$ in~\cref{lem:c_sg}.
\end{proof}

Below we give a more precise expression for the exponential factor of $c_{s,g}$.

\begin{lemma}\label{lem:exp_term}
	For $s \geq 0$ satisfying $n-s \gg 1$, we have
	$$
		\frac{\binom{2(n-s)}{n-s}}{\binom{2n-s}{n}} = \bigO(1)e^{-nf(s/n)}
	$$
	where $f(x) =  \log\left(\frac{(2-x)^{2-x}}{(4(1-x))^{1-x}} \right)$ for $x \in [0,1]$. The function $f$ satisfies $f(0) = f(1) = 0$, $f$ is strictly increasing on $[0,2/3]$ and strictly decreasing on $[2/3,1]$, and $f(x) \geq \frac{x}{2}$ for $0 \leq x \leq \frac{1}{3}$. 
\end{lemma}

\begin{proof}
	We will obtain the asymptotic expression for the ratio with Stirling's formula ($s! = (1+o(1)) \sqrt{2\pi s} (s/e)^s$). We first simplify the ratio as
    $$
        \frac{\binom{2(n-s)}{n-s}}{\binom{2n-s}{n}} = \frac{(2(n-s))!n!}{(2n-s)!(n-s)!}.
    $$
    We apply Stirling's formula to each factorial on the right-hand side individually. The total sub-exponential factor is given by 
    $$
        \sqrt{\frac{4\pi(n-s) \cdot 2\pi n}{2\pi(2n-s) \cdot 2\pi(n-s)}} \leq \sqrt{2}
    $$
    where we used $2n-s \geq n$. The exponential factor is
    \begin{align*}
        \frac{\left(\frac{2(n-s)}{e} \right)^{2(n-s)}\left(\frac{n}{e}\right)^{n}}{\left(\frac{2n-s}{e} \right)^{2n-s}\left(\frac{n-s}{e} \right)^{n-s}} =& \left(\frac{(2(1-s/n))^{2(1-s/n)}}{(2-s/n)^{2-s/n}(1-s/n)^{1-s/n}}\right)^{n}\\
        =&\left(\frac{(4(1-s/n))^{1-s/n}}{(2-s/n)^{2-s/n}}\right)^{n}\\
        =&e^{-nf(s/n)}
    \end{align*}
    This establishes the asymptotic expression in the statement of the lemma. 
    
    We can evaluate $f(0) = 0$, and we use the convention that $f(1) = \ds\lim_{x \to 1}f(x) = 0$. For $x \in [0,1)$, we have
	$$
		f'(x) = \log(1-x)-\log(2-x) + \log 4 \quad\text{ and }\quad f''(x) = \frac{1}{2-x} - \frac{1}{1-x}.
	$$
	It is easily checked that $f'(2/3) = 0$. Since $f''(x) < 0$ for $x \in [0,1)$, $f$ is concave on $[0,1)$ and hence $f$ must be strictly increasing on $[0,2/3]$ and strictly decreasing on $[2/3,1]$. Moreover, by the concavity of $f$, for $x \in [0,1/3]$ we have that $f(x) \geq \frac{f(1/3)}{1/3}\cdot x \geq \frac{x}{2}$.
\end{proof}

\begin{lemma}\label{lem:small_j} We have
	$$
		\sum_{s=1}^{0.99n}\sum_{g=2}^{\min\left\{ s+1, \frac{n-s}{2}\right\}}c_{s,g} = o(1)
	$$
	and, for any $k = k(n) \to \infty$, 
	$$
		\sum_{s=k}^{0.99n}c_{s,1} = o(1).
	$$
\end{lemma}

\begin{proof}
	Note that for $s \leq 0.99n$ and $g \leq \min\left\{s+1, \frac{n-s}{2}\right\}$, we have
	$$
		s+g \leq s + \min\left\{s+1, \frac{n-s}{2} \right\} = \min\left\{2s+1, \frac{n+s}{2} \right\}\leq \frac{1.99n}{2},
	$$
	and so $n-s-g \geq 0.005n = \frac{n}{200}$. Using Lemmas \ref{lem:djk} and \ref{lem:exp_term}, we may write for any $1 \leq s \leq 0.99n$:
	\begin{align*}
            \sum_{g=2}^{\min\left\{s+1, \frac{n-s}{2}\right\}}c_{s,g} =& \bigO(e^{-nf(s/n)})\sum_{g=2}^{\min\left\{s+1, \frac{n-s}{2}\right\}}\binom{s+1}{g}\frac{1}{(g-1)!}\left(\frac{400}{n} \right)^{g-1}\\
            =&\bigO(se^{-nf(s/n)})\sum_{g=2}^{\min\left\{s+1, \frac{n-s}{2}\right\}}\binom{s}{g-1}\frac{1}{(g-1)!}\left(\frac{400}{n} \right)^{g-1}\\
            \leq& \bigO(se^{-nf(s/n)})\sum_{g=2}^{\infty}\frac{1}{(g-1)!^{2}}\left(\frac{400s}{n} \right)^{g-1}\\
            =& \bigO\left(\frac{s^{2}}{n}e^{-nf(s/n)} \right)\sum_{g=2}^{\infty}\frac{1}{(g-1)!^{2}}\left(\frac{400s}{n} \right)^{g-2}\\
            \leq & \bigO\left(\frac{s^{2}}{n}e^{-nf(s/n)} \right)\sum_{g=2}^{\infty}\frac{400^{g-2}}{(g-1)!^{2}}\\
            =&\bigO\left(\frac{s^{2}}{n}e^{-nf(s/n)} \right).
	\end{align*}
	Thus, to show the first statement of the Lemma, it suffices to show that the sum $\sum_{s=1}^{0.99n}\frac{s^{2}}{n}e^{-nf(s/n)}$ is $o(1)$. For $1 \leq s \leq \frac{n}{3}$, by \cref{lem:exp_term} we have that $f(s/n) \ge s/2n$ and so we may write
	$$
		\sum_{s=1}^{n/3}\frac{s^{2}}{n}e^{-nf(s/n)} \leq \sum_{s=1}^{n/3}\frac{s^{2}}{n}e^{-s/2} = \bigO(1/n) = o(1).
	$$
	For $\frac{n}{3} \leq s \leq 0.99n$, we have that $f(s/n) \geq \min\{f(1/3), f(0.99)\} = f(0.99) = 0.04223...\geq 0.04$, and hence 
	$$
		\sum_{s=n/3}^{0.99n}\frac{s^{2}}{n}e^{-nf(s/n)} = \bigO(n^{2}e^{-0.04n}) = o(1).
	$$
	This proves the first statement in the Lemma. For the second, again using Lemmas \ref{lem:djk} and \ref{lem:exp_term} we can write 
	$$
		c_{s,1} \leq \bigO\left(se^{-nf(s/n)} \right)
	$$
	for $1 \leq s \leq 0.99n.$ Note that, similar to above, we have 
    $$
        \sum_{s=n/3}^{0.99n}se^{-nf(s/n)} \leq \bigO(n^{2}e^{-0.04n}) = o(1).
    $$
    So the second statement holds for any $k = k(n)$ in $[n/3, 0,99n].$ Suppose $\omega(1) = k \leq n/3$. Then, we have $\sum_{s=k}^{n/3}se^{-nf(s/n)} \leq \sum_{s=k}^{n/3}se^{-s/2}$. Note that the ratio of successive terms in the sum is $(s+1)e^{-(s+1)/2}/se^{-s/2} = \frac{s+1}{s}e^{-1/2} = (1+o(1))e^{-1/2} < 0.61$, since $s \ge k \gg 1$. Thus,
    $$
        \sum_{s=k}^{\infty}se^{-s/2} \leq ke^{-k/2}\sum_{\ell = 0}^{\infty}(0.61)^{\ell} = \bigO \left( ke^{-k/2} \right) = o(1),
    $$
    which completes the proof of the second statement.
\end{proof}

\begin{lemma}\label{lem:large_j}
    We have
	$$
		\sum_{s=0.99n}^{n-2}\sum_{g=1}^{\min\left\{s+1,\frac{n-s}{2}\right\}}c_{s,g} 
        = o(1).
	$$
\end{lemma}

\begin{proof}
	Clearly we have $\min\left\{s+1, \frac{n-s}{2}\right\} = \frac{n-s}{2}$ for $0.99n \leq s \leq n-2$. We write
	$$
            \frac{\binom{2(n-s)}{n-s}}{\binom{2n-s}{n}} \leq \frac{4^{n-s}}{(n+n-s)_{(n-s)}/(n-s)!} \leq \left(\frac{4}{n} \right)^{n-s}(n-s)!,
	$$
        where $(n)_k=n\cdot(n-1)\cdots(n-k+1)\ge(n-k)^k$ denotes the falling factorial.
	For $s \leq n-2$ and $g \leq \frac{n-s}{2}$, we have $n-s-g \geq 1$. Thus by \cref{lem:djk} and the above inequality, 
	\begin{eqnarray*}
            \sum_{g=1}^{\frac{n-s}{2}}c_{s,g} &\leq&  \left(\frac{4}{n} \right)^{n-s}(n-s)!\sum_{g=1}^{\frac{n-s}{2}}\binom{s+1}{g}\frac{2^{g-1}}{(g-1)!}\\
            &\leq& \left(\frac{4}{n} \right)^{n-s}(n-s)!\binom{n}{\frac{n-s}{2}}\sum_{g=1}^{\infty}\frac{2^{g-1}}{(g-1)!}\\
            &\leq& e^{2} \left(\frac{4}{n} \right)^{n-s}(n-s)!\frac{n^{\frac{n-s}{2}}}{\left(\frac{n-s}{2} \right)!}\\
            &\leq& e^{2}\left(\frac{16(n-s)}{n} \right)^{\frac{n-s}{2}}.
	\end{eqnarray*}
	Then we have
	$$
		\sum_{s=0.99n}^{n-2}c_{s,g} \leq e^{2}\sum_{s=0.99n}^{n-2}\left(\frac{16(n-s)}{n} \right)^{\frac{n-s}{2}} = e^{2}\sum_{\ell = 2}^{0.01n}\left(\frac{16\ell}{n} \right)^{\ell/2} = o(1),
	$$
	and the proof of the lemma is finished.
\end{proof}
\begin{lemma}\label{lem:dominant_term}
	For any $k = k(n) \to \infty$, 
	$$
		\sum_{s=0}^{k}c_{s,1} = 4 + o(1).
	$$
\end{lemma}

\begin{proof}	
    We remark that $c_{s,1}$ is given by
    \begin{equation}\label{eq:cs1}
        c_{s,1}=(s+1)\frac{n!}{(2n-s)!}\binom{2(n-s)}{n-s}w_{n-s}.
    \end{equation}
    by \cref{lem:c_sg}. It will suffice to prove the result for $k = o(\sqrt{n})$. Indeed, given this, for arbitrary $k \to \infty$ we may choose $h \leq k$ with $h = o(\sqrt{n})$; then 
    $$
        \sum_{s=0}^{k}c_{s,1} = \sum_{s=0}^{h}c_{s,1} + \sum_{s=h+1}^{k}c_{s,1} = 4 + o(1),
    $$
    where we use \cref{lem:small_j} to conclude that $\sum_{s=h+1}^{k}c_{s,1} = o(1).$
    
	Now, let $k = o(\sqrt{n})$. We claim that $w_{n-s} = (1+o(1))(n-s)!$, uniformly over $1 \leq s \leq k.$ To see this, first observe that $w_{n-s}$ is lower bounded by the number of paths of length $n-s$ from $\emptyset$ to $[n-s]$ which are \textit{vertex disjoint} from $\pi_{n-s}$ on the interior of the $n-s$-dimensional hypercube, i.e., vertex disjoint except at $\emptyset$ and $[n-s]$. Suppose we sample a path of length $n-s$ from $\emptyset$ to $[n-s]$ uniformly at random. The probability that the path goes through any particular vertex on level $h$ of the hypercube is $\binom{n-s}{h}^{-1}$. Hence, by a union bound, the probability that the path intersects with $\pi_{n-s}$ on the interior of a $n-s$-dimensional hypercube is at most $\sum_{h=1}^{n-1}\binom{n-s}{h}^{-1}= \bigO((n-s)^{-1}) = \bigO(n^{-1}).$ It follows that there are $(1-\bigO(n^{-1}))(n-s)!$ paths which are vertex disjoint from $\pi_{n-s}$ on the interior, so we have $(1+o(1))(n-s)! \leq w_{n-s}$. Trivially, $w_{n-s} \leq (n-s)!$, and this establishes the claim.

    We plug $w_{n-s} =(1+o(1))(n-s)!$ into \eqref{eq:cs1} to reach
    $$
        c_{s,1} = (1+o(1))(s+1)\frac{n!}{(n-s)!}\cdot\frac{(2(n-s))!}{(2n-s)!}
    $$
    for $1 \leq s \leq k$. Since $k = o(\sqrt{n})$, we have $\frac{n!}{(n-s)!} = (1+o(1))n^{s}$ and $\frac{(2(n-s))!}{(2n-s)!} = (1+o(1))\frac{1}{(2n)^{s}}.$ Thus in total we get $c_{s,1} = (1+o(1))\frac{s+1}{2^{s}}$ uniformly over $1 \leq s \leq k$, so that finally
	$$
		\sum_{s=0}^{k}c_{s,1} = (1+o(1))\sum_{s=0}^{k}\frac{s+1}{2^{s}} = 4 + o(1),
	$$
    and the proof of the lemma is finished.
\end{proof}

\section{Proofs for~\cref{sec:trees}}\label{sec:trees_appendix}

\begin{proof}[Proof of~\cref{lem:coupling-admissible}]
    We bound $M_{k,r}(\varepsilon)$ in order to derive an admissible sequence $n_{k,r}$.
    First we bound
    \begin{align*}
        M_{k,r}(\varepsilon)\le \max_{\ell\in[k],\bm i\in[r]^\ell}\left\{(\ell r+Y_{\varepsilon}(\bm i'))\right\}\cdot \left(1+\frac1\varepsilon\max_{\ell\in[k],\bm i\in [r]^\ell}\Delta(\bm i)\right).
    \end{align*}
    The parameter of the Poisson random variable $Y_{\varepsilon}(\bm i)$ is $|P(\bm i)|\varepsilon$ plus a sum of exponential random variables. Recall that a mixed Poisson distribution with exponential mixture follows the geometric distribution, while a sum of geometric distributions follows a negative binomial distribution. This allows us to decompose $Y_{\varepsilon}(\bm i)$ into $Y(\bm i)\sim\text{Poi}(|P(\bm i)|\varepsilon)$ and $B(\bm i)\sim \text{NB}(|P(\bm i)|,\frac12)$, where NB denotes the negative binomial distribution; the number of failed trials until the $|P(\bm i)|$-th success.
    This leads to
    \begin{align*}
        M_{k,r}(\varepsilon)\le \max_{\ell\in[k],\bm i\in[r]^\ell}\left\{(\ell r+Y(\bm i')+B(\bm i'))\right\}\cdot \left(1+\frac1\varepsilon\max_{\ell\in[k],\bm i\in [r]^\ell}\Delta(\bm i)\right).
    \end{align*}
    We will pick
    \[
    n_{k,r}=(kr+y_k+b_k)(1+\varepsilon^{-1}d_k),
    \]
    where $y_k,b_k,d_k$ are chosen in order to guarantee 
    \[
    \P\left(\max_{\ell\in[k],\bm i\in [r]^\ell}Y(\bm i')> y_k\right),\P\left(\max_{\ell\in[k],\bm i\in [r]^\ell}B(\bm i')> b_k\right),\P\left(\max_{\ell\in[k],\bm i\in [r]^\ell}\Delta(\bm i)> d_k\right)\le \frac1{3k^2},
    \]
    so that
    \(
    \P(M_{k,r}(\varepsilon)>n_{k,r})\le\frac1{k^2}.
    \)
    We will now derive the desired asymptotics of $y_k,b_k,d_k$.
    Firstly note that
    \[
    \left|\bigcup_{\ell\in[k]}[r]^\ell\right|=\frac{r^{k+1}-1}{r-1}\sim r^k.
    \]
    For $y_k$, we use the fact that $Y(\bm i')$ is stochastically dominated by $\text{Poi}(rk\varepsilon)$ and use the Chernoff inequality to bound
    \[
    \P(Y(\bm i')>y_k)\le \left(\frac{erk\varepsilon}{y_k}\right)^{y_k}e^{-rk\varepsilon}.
    \]
    Then, by the union bound, it suffices to have
    \[
    \frac{r^{k+1}-1}{r-1}\left(\frac{erk\varepsilon}{y_k}\right)^{y_k}e^{-rk\varepsilon}\le\frac1{3k^2}.
    \]
    We pick $y_k$ to satisfy this constraint with equality. We take the log to determine the asymptotics of $y_k$.
    \[
    k\log r+o(1)+y_k\left(1+\log r+\log k-\log y_k\right)-rk\varepsilon=-\log3-2\log k.
    \]
    The main asymptotics are determined by 
    \[
    k\log r\sim y_k\left(\log y_k-\log k\right)\Rightarrow y_k\sim k\frac{\log r}{\log\log r}.
    \]

    For $b_k$, we make use of the fact that each of the $B(\bm i')$ is stochastically dominated by $\text{NB}(rk,\frac12)$. In addition, the event that the $rk$-th failure occurs before the $a_k$-th trial is equivalent to the event that the first $a_k$ trials contain less than $rk$ successes. This allows us to use Hoeffding's inequality to bound
    \begin{align*}
    \P(B(\bm i')+rk>a_k)\le \P(\text{Bin}(a_k,\tfrac12)<rk)\le e^{-2a_k\left(\frac12-\frac{rk}{a_k}\right)^2}.
    \end{align*}
    Again, by the union bound it suffices to have
    \[
    \frac{r^{k+1}-1}{r-1}e^{-2a_k\left(\frac12-\frac{rk}{a_k}\right)^2}=\frac1{3k^2}.
    \]
    We pick $a_k$ to satisfy this constraint with equality and set $b_k=a_k-rk$. We inspect the asymptotics of $a_k$:
    \[
    k\log r+o(1)-2a_k\left(\frac12-\frac{rk}{a_k}\right)^2=\bigO(1)-2\log k.
    \]
    We rewrite this to
    \[
    \frac{a_k}{2}+\frac{2r^2k^2}{a_k}=k\log r+2\log k+2rk+\bigO(1)\sim 2rk,
    \]
    which leads to $a_k\sim 2rk$, so that $b_k\sim rk$.
    
    Finally, we choose $r_k$ so that
    \[
    \frac{r^{k+1}-1}{r-1}e^{-d_k}=\frac1{3k^2},
    \]
    which has asymptotics $d_k\sim k\log r$.
    Putting everything together, we obtained
    \[
    n_{k,r}=(kr+y_k+b_k)(1+\varepsilon^{-1}d_k)\sim 2kr\cdot \varepsilon^{-1}k\log r=2k^4r^2\log r,
    \]
    for $\varepsilon=k^{-2}r^{-1}$.
\end{proof}

\begin{proof}[Proof of~\cref{lem:trees}]
    We take $Z_k=Z_{n,k,r_k}$ as defined in \cref{lem:exp-limit1} and its inverted counterpart $Z_k'$ (see proof of~\cref{lem:exp-limit}).

    We write
    \[
    \lambda_{k,n}=\sum_{\pi\in\Pi_k^*}\E_k[I_\pi].
    \]
    Recall that $W_k(\pi)$ is the $k$-th weight of $\pi$.
    Note that
    \[
    \E_k[I_\pi]=\frac{(W_{n-k+1}(\pi)-W_k(\pi)))^{n-2k}}{(n-2k)!},
    \]
    since each of the $n-2k$ weights need to be in $(W_k(\pi),W_{n-k+1}(\pi))$, and they need to be ordered correctly.

    We denote the set of level-$k$ vertices that are reachable from $\emptyset$ via a tree path by $V_k^*=\{v_k(\pi)\ :\ \pi\in\Pi^*_k\}$. Similarly, we will write $V_{n-k}^*=\{v_{n-k}(\pi)\ :\ \pi\in\Pi^*_k\}$ to denote the set of level-$(n-k)$ vertices from which there is a tree path to $[n]$.
    Each such vertex $u\in V_k^*$ corresponds to a path $\pi\in \Pi^*_k$. We let $w(u)=W_k(\pi)$ denote the weight of the edge leading to $u$ in $\pi$. We also define $w'(u')=1-W_{n-k+1}(\pi)$ for $u'\in V_{n-k}^*$, so that the distribution $w'(U')$ is equal to the distribution of $w(U)$ for $U\in V_k^*$ and $U'\in V_{n-k}^*$ drawn uniformly at random.

    This leads to the alternative expression
    \begin{align*}
    \lambda_{k,n}
    &=\sum_{u\in V_k^*}\sum_{u'\in V_{n-k}^*}\I{u\subset u',w(u)+w'(u')\le 1}\sum_{\pi\in\Pi(u,u')}\frac{(1-w(u)-w'(u'))^{n-2k}}{(n-2k)!}\\
    &=\sum_{u\in V_k^*}\sum_{u'\in V_{n-k}^*}\I{u\subset u',w(u)+w'(u')\le 1}(1-w(u)-w'(u'))^{n-2k},
    \end{align*}
    where the indicator is needed to ensure the possibility of an increasing path from $u$ to $u'$. The final step follows from the fact that $|\Pi(u,u')|=(|u'|-|u|)!$ whenever $u\subset u'$.

    Similarly, $Z_kZ_k'$ can be written as
    \[
    Z_kZ_k'=\sum_{u\in V_k^*}\sum_{u'\in V_{n-k}^*}\left(1-w(u)\right)^n\left(1-w'(u')\right)^n,
    \]
    and we need to show that the difference between these expressions vanishes in probability.

    We draw the vertices $U\in V_k^*$ and $U'\in V_{n-k}^*$ uniformly.
    Note that $|V_k^*|=|V_{n-k}^*|=r_k^k$. This allows us to write
    \[
    \lambda_{k,n}-Z_kZ_k'=r_k^{2k}\E_k[\I{U\subset U',w(U)+w'(U')\le 1}(1-w(U)-w'(U'))^{n-2k}-(1-w(U))^n(1-w'(U'))^n].
    \]
    Note that every vertex $v\in V_k$ is equally likely to be in $V_k^*$ (i.e., to have a tree path). This means that $U$ is uniform over $V_k$, and $U'$ is uniform over $V_{n-k}$. Furthermore, both trees are independent since they depend on a disjoint set of weights. Hence, $U$ and $U'$ are independent.
    
    To prove that the difference between these random variables vanishes, we separate this difference into several nonnegative components and prove that each of their expectations vanishes.
    Firstly, we bound the influence of the indicator
    \[
    \varepsilon_1=Z_kZ_k'-r_k^{2k}\E_k[\I{U\subset U',w(U)+w'(U')\le1}\cdot\left(1-w(U)\right)^n\left(1-w'(U')\right)^n]\ge0.
    \] 
    Secondly, we use \[(1-w(U))(1-w'(U'))=1-w(U)-w'(U')+w(U)w'(U')\ge1-w(U)-w'(U')\]
    and define
    \[
    \varepsilon_2=r_k^{2k}\E_k\left[\I{U\subset U',w(U)+w'(U')\le1}\cdot\left(\left(1-w(U)\right)^n\left(1-w'(U')\right)^n-\left(1-w(U)-w'(U')\right)^n\right)\right]\ge0.
    \]
    Finally, we consider the change in exponent
    \[
    \varepsilon_3=r_k^{2k}\E_k\left[\I{U\subset U',w(U)+w'(U')\le1}\left(\left(1-w(U)-w'(U')\right)^{n-2k}-\left(1-w(U)-w'(U')\right)^n\right)\right]\ge0,
    \]
    so that
    \[
    \lambda_{k,n}-Z_kZ_k'=\varepsilon_3-\varepsilon_2-\varepsilon_1.
    \]
    Since each of these errors is nonnegative, we have
    \[
    |\lambda_{k,n}-Z_kZ_k'|\le\varepsilon_1+\varepsilon_2+\varepsilon_3,
    \]
    so that it suffices to prove that each of these errors vanish in expectation.

    \subsubsection*{Vanishing $\varepsilon_1$.} 
    Note that $\E[\varepsilon_1]\le\E[Z_kZ_k']\le1$, so that by the dominated convergence theorem, we only need to show that $\P(U\not\subset U'\text{ or }w(U)+w'(U')>1)\to0$.
    By the independence of $U$ and $U'$, we have
    \[
        \P(U\subset U')=\frac{{n-k\choose k}}{{n\choose k}}\sim \frac{(n-k)^k/k!}{n^k/k!}=\left(1-\frac kn\right)^{k}\to1
    \]
    for $k\ll\sqrt{n}$.
    Furthermore, by~\cref{lem:poi-coupling}, we have
    \[
    -n\log(1-w(u))\le \frac1k+\sum_{\bm j\in P(\bm i(u))}\Delta(\bm j),
    \]
    for all $u\in V_k^*$ with probability at least $1-k^{-2}$, where $\bm i(u)$ is the index corresponding to $u\in V_k^*$. Together with $-n\log(1-w(u))\ge nw(u)$ and the Markov inequality, we obtain the bound
    \begin{align*}
    \P(w(U)>\tfrac12)&\le k^{-2}+\P\left(\frac1k+\sum_{\bm j\in P(\bm i(U))}\Delta(\bm j)>\frac n2\right)\\
    &\le k^{-2}+\frac2n\E\left[\frac1k+|P(\bm i(U))|\right]\\
    &\le k^{-2}+\frac2{kn}+\frac{2kr}n,
    \end{align*}
    which vanishes for $1\ll k\ll\sqrt{n}$.
    We conclude
    \[
    \P(U\not\subset U'\text{ or }w(U)+w'(U')>1)\le \P(U\not\subset U')+\P(w(U)>\tfrac12)+\P(w(U')>\tfrac12)\to0,
    \]
    so that $\E[\varepsilon_1]\to0$.
    \subsubsection*{Vanishing $\varepsilon_2$.}
    We rewrite
    \begin{align*}
        &(1-w(U))^n(1-w'(U'))^n-(1-w(U)-w'(U'))^n\\
        =&(1-w(U))^n(1-w'(U'))^n\left(1-\left(1-\frac{w(U)}{1-w(U)}\frac{w'(U')}{1-w'(U')}\right)^n\right)\\
        =&\bigO\left(nw(U)(1-w(U))^{n-1}w'(U')(1-w'(U'))^{n-1}\right).
    \end{align*}
    Now, note that $w(U)(1-w(U))^{n-1}$ is maximized at $w(U)=1/n$, with value $\tfrac1n(1-\tfrac1n)^{n-1}\le \tfrac1n$, so that the above is $\bigO(n^{-1})$.
    This tells us that
    \[
    \E[\varepsilon_2]=\bigO(r_k^{2k}/n),
    \]
    which vanishes for $k\ll\frac{\log n}{\log\log n}$.
    \subsubsection*{Vanishing $\varepsilon_3$.} 
    We write
    \begin{align*}
    \E[\varepsilon_3]
    &=r_k^{2k}\E\left[\I{U\subset U',w(U)+w'(U')\le1}\left(\left(1-w(U)-w'(U')\right)^{n-2k}-\left(1-w(U)-w'(U')\right)^{n}\right)\right]\\
    &\le r_k^{2k}\E\left[\left(\left(1-w(U)-w'(U')\right)^{n-2k}-\left(\left(1-w(U)-w'(U')\right)^{n-2k}\right)^{\frac n{n-2k}}\right)\right].
    \end{align*}
    Using concavity of $y\mapsto -y^{\frac n{n-2k}}$, Jensen's inequality yields
    \begin{align*}
    \E[\varepsilon_3]&\le r_k^{2k}\left(\E\left[(1-w(U)-w'(U'))^{n-2k}\right]-\E\left[(1-w(U)-w'(U'))^{n-2k}\right]^{\frac n{n-2k}}\right)\\
    &=r_k^{2k}\E\left[(1-w(U)-w'(U'))^{n-2k}\right]\left(1-\E\left[(1-w(U)-w'(U'))^{n-2k}\right]^{\frac{2k}{n-2k}}\right).
    \end{align*}
    Next, using $1-e^{-x}\le x$ for $x=-\frac{2k}{n-2k}\log\E_k\left[(1-w(U)-w'(U'))^{n-2k}\right]\ge0$, we obtain
    \begin{align*}
        \E[\varepsilon_3]
        &\le -\frac{2kr_k^{2k}}{n-2k}\E\left[(1-w(U)-w'(U'))^{n-2k}\right]\log\E\left[(1-w(U)-w'(U'))^{n-2k}\right]\\
        &=\frac{4k^2\log r_k}{n-2k}-\frac{2k}{n-2k}q\log q,
    \end{align*}
    for $q=r_k^{2k}\E\left[(1-w(U)-w'(U'))^{n-2k}\right]$.
    The nonnegativity of $\varepsilon_3$ implies $q\log q\le 2k\log k$, so that we only need the first term to vanish, which occurs whenever $k^2\log r_k\ll n$,
    and the proof of the lemma is finished.
\end{proof}

\section{The limiting distribution}
\begin{lemma}\label{lem:gompertz}
    The limiting distribution is given by
    \[
    \P(X=x)=\sum_{k=0}^x{x\choose k}\frac{\delta-\sum_{r=0}^{k-1}(-1)^rr!}{k!},
    \]
    where $\delta=\int_{0}^{\infty}\frac{e^{-z}}{1+z}\,dz$ is the \emph{Gompertz} constant.
\end{lemma}
\begin{proof}
    We will prove this theorem by rewriting
    \[
    \P(X=x)=\E\left[e^{-ZZ'}\frac{(ZZ')^x}{x!}\right]
    \]
    to the desired formula. We first consider the conditional expectation
    \begin{align*}
    \E\left[e^{-ZZ'}\frac{(ZZ')^x}{x!}\ \middle|\ Z=z\right]
    &=\E\left[e^{-zZ'}\frac{(zZ')^x}{x!}\right]\\
    &=\frac{z^x}{x!}\int_0^\infty (z')^xe^{-zz'-z'}dz'\\
    &=\frac{z^x}{(z+1)^{x+1}},
    \end{align*}
    where the last step follows by partial integration and solving the resulting recurrence.
    Taking the expectation over $Z$, we obtained,
    \begin{equation}\label{eq:gompertz1}
    \P(X=x)=\E\left[\frac{Z^x}{(1+Z)^{x+1}}\right]=\E\left[\frac1{Z+1}\left(1-\frac1{Z+1}\right)^x\right]=\sum_{k=0}^x{x\choose k}(-1)^k\E\left[(Z+1)^{-k-1}\right].
    \end{equation}
    We thus need to compute negative moments
    \[
    m_k=\E\left[(Z+1)^{-k}\right]=\int_0^\infty\frac{e^{-z}dz}{(1+z)^k},
    \]
    for $k\ge1.$ Note that $m_1=\delta$ is Gompertz constant.
    We apply integration by parts to obtain the recursion
    \[
    m_{k+1}=\int_0^\infty\frac{e^{-z}dz}{(1+z)^{k+1}}=\left[-\frac1k\frac{e^{-z}}{(1+z)^r}\right]_{z=0}^\infty-\frac1k\int_0^\infty\frac{e^{-z}dz}{(1+z)^k}=\frac{1-m_k}{k}.
    \]
    The above can be rewritten to $k!m_{k+1}=(k-1)!-(k-1)!m_k$. Repeating this relation leads to
    \[
    k!m_{k+1}=\sum_{r=0}^{k-1}(-1)^r(k-1-r)!+(-1)^k\delta,
    \]
    so that
    \[
    m_{k+1}=\frac{(-1)^k\delta+\sum_{r=0}^{k-1}(-1)^{k-1-r}r!}{k!}=(-1)^k\frac{\delta-\sum_{r=0}^{k-1-r}(-1)^{r}r!}{k!}
    \]
    Plugging this into~\eqref{eq:gompertz1} yields the desired expression.
\end{proof}

\end{document}